\documentclass[onefignum,onetabnum]{siamonline250211}

\usepackage[caption=false]{subfig}
\usepackage[noend]{algpseudocode}
\usepackage{tabularx}
\usepackage{booktabs}
\usepackage{layouts}

\usepackage{math}

\let\oldthebibliography=\thebibliography
\let\oldendthebibliography=\endthebibliography

\usepackage{natbib}
\renewcommand{\cite}{\citep}

\let\thebibliography=\oldthebibliography
\let\endthebibliography=\oldendthebibliography

\newcommand{\SDPLRp}{\texttt{SDPLR+}\xspace}
\newcommand{\LRSDP}{\texttt{LRSDP}\xspace}
\newcommand{\kmeanspp}{\texttt{k-means++}\xspace}
\newcommand{\Qlqh}{$\text{Q}^{\text{lqh}}$\xspace}
\newcommand{\Qlse}{$\text{Q}^{\text{lse}}$\xspace}
\newcommand{\nnz}{\text{nnz}\xspace}
\newcommand{\SketchyCGAL}{\texttt{SketchyCGAL}\xspace}
\newcommand{\CSDP}{\texttt{CSDP}\xspace}

\newcommand{\gset}{\textsc{Gset}\xspace}
\newcommand{\snap}{\textsc{Snap}\xspace}
\newcommand{\dimacs}{\textsc{Dimacs}\xspace}
\newcommand{\dimacsten}{\textsc{Dimacs}$10$\xspace}
\renewcommand{\S}{\ensuremath{\mathbb{S}}}
\newcommand{\vol}{\mathrm{Vol}}
\newcommand{\cut}{\textbf{cut}}
\renewcommand{\diag}{\text{diag}}
\providecommand{\Diag}{\text{Diag}}

\algnewcommand{\IfThenElse}[3]{
  \State \algorithmicif\ #1\ \algorithmicthen\ #2 \State \algorithmicelse\ #3}

\algrenewcommand\algorithmicrequire{\textbf{Input:}}
\algrenewcommand\algorithmicensure{\textbf{Output:}}

\usepackage{lipsum}
\usepackage{amsfonts}
\usepackage{graphicx}
\usepackage{epstopdf}
\ifpdf
  \DeclareGraphicsExtensions{.eps,.pdf,.png,.jpg}
\else
  \DeclareGraphicsExtensions{.eps}
\fi

\usepackage{enumitem}
\setlist[enumerate]{leftmargin=.5in}
\setlist[itemize]{leftmargin=.5in}

\newsiamremark{remark}{Remark}
\newsiamremark{hypothesis}{Hypothesis}
\crefname{hypothesis}{Hypothesis}{Hypotheses}
\newsiamthm{claim}{Claim}
\newsiamremark{fact}{Fact}
\crefname{fact}{Fact}{Facts}

\headers{A faster and scalable low-rank SDP solver SDPLR+} {Yufan
  Huang and David F. Gleich}

\title{Suboptimality bounds for trace-bounded SDPs enable a faster and
  scalable low-rank SDP solver SDPLR+\thanks{Submitted to the editors
    DATE.  \funding{Gleich and Huang acknowledge partial support from
      awards NSF IIS-2007481, DOE DE-SC0023162, as well as the IARPA
      AGILE Program.}}}

\author{Yufan Huang\thanks{Purdue University, West Lafayette,
    IN (\email{2019hyf@gmail.com}).}  \and David
  F. Gleich\thanks{Purdue University, West Lafayette, IN
    (\email{dgleich@purdue.edu}).}  }

\usepackage{amsopn}

\ifpdf
\hypersetup{
  pdftitle={Suboptimality bounds for trace-bounded SDPs enable a faster and
  scalable low-rank SDP solver SDPLR+},
  pdfauthor={Y. Huang and D. F. Gleich}
}
\fi
\headers{A faster and scalable low-rank SDP solver SDPLR+} {Yufan
  Huang and David F. Gleich}

\begin{document}

\maketitle

\begin{abstract}
  Semidefinite programs (SDPs) and their solvers are powerful tools
  with many applications in machine learning and data
  science. Designing scalable SDP solvers is challenging because by
  standard the positive semidefinite decision variable is an
  $n \times n$ dense matrix, even though the input is often an
  $n \times n$ sparse matrix. However, the solution may not require a
  full-rank matrix, as shown by Barvinok and Pataki.  Two decades ago,
  Burer and Monteiro developed an SDP solver \texttt{SDPLR} that
  optimizes over a low-rank factorization instead of the full matrix.
  This greatly decreases the storage cost and works well for many
  problems.  The original solver \texttt{SDPLR} tracks only the primal
  infeasibility of the solution, preventing early termination at
  moderate accuracy. We use a suboptimality bound for trace-bounded
  SDP problems that enables us to track the progress better and
  perform early termination. We then develop \texttt{SDPLR+}, which
  starts the optimization with an extremely low-rank factorization and
  dynamically updates the rank based on the primal infeasibility and
  suboptimality. This further speeds up the computation and saves
  storage.  Numerical comparisons on Max Cut, Minimum Bisection, Cut
  Norm, and Lov\'{a}sz Theta problems with many recent
  memory-efficient scalable SDP solvers demonstrate the scalability of
  \texttt{SDPLR+} up to problems with million-by-million decision
  variables. It is often the fastest solver to a moderate accuracy of
  $10^{-2}$.  Further experiments on $\mu$-conductance, matrix
  completion, and $k$-means clustering show the potential of
  \texttt{SDPLR+} on a broader range of data science applications.
\end{abstract}



\section{Introduction}
Semidefinite programs (SDPs) are an extremely capable class of convex
optimization problems. Conceptually, they generalize the class of
\emph{vector} decision variables used in linear programs to
\emph{matrix} decision variables. This creates optimization problems
where the solution is a symmetric, positive semidefinite matrix. The
class of semidefinite programs has many applications in machine
learning, including matrix completion~\cite{CRExact2012}, $k$-means
clustering~\cite{kulis-2007,Whang-2019-neo-k-means}, combining SAT
solvers and deep learning~\cite{wang19e}, along with many more
applications in control theory~\cite{BoyLinear1994}. In addition, SDPs
have offered a systematic way of designing approximation algorithms
for many NP-hard combinatorial problems, for instance, graph cut
problems~\cite{GWImproved1995,ARVExpander2009,HSGTheoretical2023},
graph clustering, and matrix cut-norm~\cite{ANApproximating2004}.

We consider SDPs of the following linear form
\begin{align}
  \label{eq:4} 
  \minimize_{\mX \in \S^n} \; \langle \mC, \mX \rangle \; \subjectto \; \cA(\mX) = \vb, \mX \succeq 0 \tag{SDP} 
\end{align}
where \(\S\) is the set of symmetric matrices with size
\(n \times n\), \(\mC \in \S^n\) is a cost matrix,
\(\langle \cdot, \cdot \rangle\) denotes the inner product of two matrices
\(\langle \mA, \mB \rangle = \smash{\sum_{ij} A_{ij} B_{ij}} \), and
\(\cA: \S^n \to \R^m\) is a linear operator encoding \(m\) linear
constraints with right-hand side vector \(\vb\). The constraint
notation \(\mathcal{A}(\mX)\) corresponds to the vector
\([ \langle \mA_1, \mX\rangle, \ldots , \langle\mA_m, \mX \rangle ]^{T}\). As a convex problem,
\cref{eq:4} can be solved to high accuracy via interior point methods
on small problem
instances~\cite{BorCSDP1999,StuUsing1999,TTTSDPT31999}. Interior point
methods, however, cannot solve large problem instances due to the
memory needed to store the decision variable with \(n^2\) entries and
the second-order information. This property has made scaling SDPs to
large problem instances, where $n$ is in the millions, challenging
and, consequently, has focused a long and continuing thread of
research on scaling SDPs.

Among scalable SDP solvers, possibly the most famous one is the Burer
and Monteiro solver
\texttt{SDPLR}~\cite{BMNonlinear2003,BMLocal2005,BCComputational2006}.
It was developed based on the observation that \cref{eq:4} admits an
optimal solution with rank \(O(\sqrt{m})\) based on a bound by
Barvinok and Pataki~\cite{BarProblems1995,PatRank1998}. Therefore,
they factorize the decision variable \(\mX\) into \(\mY\mY^T\) with
\(\mY\) having size \(n \times r\) where \(r = \Theta(\sqrt{m})\) and transform
\cref{eq:4} into the following \cref{eq:7}.
\begin{align}
  \label{eq:7}
  \minimize_{\mY \in \R^{n \times r}} \; \langle \mC, \mY\mY^T \rangle \; \subjectto \;  \cA(\mY\mY^T) = \vb \tag{BM-SDP} 
\end{align}
This transformation naturally eliminates the hard-to-optimize positive
semidefiniteness constraint and \texttt{SDPLR} tackles \cref{eq:7}
using the augmented Lagrangian (ALM) framework. The drawback of this
transformation is that it leads to a nonconvex problem and the
convergence to the global optimum is not guaranteed. Recent
research~\cite{BVBNonconvex2018,CifBurer2021} has shown that for a
generic problem instance, any local optimum of \cref{eq:7} is a
global optimum when \(r = \Omega(\sqrt{m})\), and non-generic exceptions
are possible~\cite{OSVBurerMonteiro2022}.

Also recently, another line of research studies memory-efficient SDP
solvers based on the Frank-Wolfe
method~\cite{JagRevisiting}. Regardless of how the Frank-Wolfe method
is integrated into the SDP solvers, it typically requires an
easy-to-characterize compact domain for solving the linear
optimization subproblems
efficiently~\cite{HazSparse2008,YTF+Scalable2021a,SNSMemoryEfficient2021a,SNSMemoryEfficient2021,PGSScalable2023,YFCConditionalGradientBased2019}. Often
this is a class of trace-bounded SDPs
\begin{align}
\label{eq:3}
  \minimize_{\mX \in \Delta_{\alpha}} \; \langle \mC, \mX \rangle \;  \subjectto \; \mathcal{A}(\mX) = \vb \tag{Trace-Bounded SDP} 
\end{align}
where
$\Delta_{\alpha} = \left\{ \mX \succeq 0: \text{Tr}(\mX) \leq \alpha \right\}$ is the convex
cone of positive semidefinite matrices with trace bounded by a
constant $\alpha \geq 0$. For this set, the conic linear optimization subproblem
involved in Frank-Wolfe turns into an eigenvalue problem. One can see
\cref{eq:3} has the same optimum with \cref{eq:4} if we pick
$\alpha \geq \Tr(\mX^*)$ where $\mX^{*}$ is an optimal solution. While
\cref{eq:3} is only a subclass of \cref{eq:4}, many SDPs used in
applications directly have a trace bound encoded by the constraints,
or have a trace bound that can be estimated. Examples include
$k$-means~\cite{kulis-2007,Whang-2019-neo-k-means}, graph
cuts~\cite{HSGTheoretical2023,FJImproved1997a}, phase
retrieval~\cite{balan2009painless}, correlation
clustering~\cite{bansal2004correlation}, and those we test in the
experiments.

One advantage of these Frank-Wolfe methods over \texttt{SDPLR} is that they
have an easy-to-compute surrogate duality bound for tracking the
suboptimality. This enables them to track the progress more precisely
and terminate the optimization efficiently.  Numerical experiments
demonstrate their scalability on large-scale problems, especially for
moderately accurate
solutions~\cite{PGSScalable2023,YTF+Scalable2021a}.  We note that
\texttt{SDPLR} does not track the suboptimality because the dual
problem of \cref{eq:4} (i.e.~\cref{eq:5}, which will be stated shortly) 
has a feasible region with a tricky constraint that requires 
a matrix to be positive semidefinite.

Our core contribution in this paper is that we design a faster and
more scalable SDP solver called \texttt{SDPLR+} based on combining
\texttt{SDPLR} with two techniques.
\begin{itemize}
\item We use the fact that the trace bound in \cref{eq:3} 
  results in an unconstrained dual problem where strong duality always
  holds (\cref{thm:strong-duality}).  This enables any solver
  with both primal and dual estimates to have a suboptimality bound,
  including \texttt{SDPLR} (\cref{sec:suboptimality}).
\item The suboptimality bound and primal infeasibility bound allow us
  to design a solver tolerant of smaller rank iterates.  Since we can
   track progress better, we start the
  optimization from an extremely small rank parameter $r$ and
  dynamically update the rank when no progress is made after a while.  
  This speeds up the computation and reduces the
  memory required.
\item The code is available at
  \url{https://github.com/luotuoqingshan/SDPLRPlus.jl}.
\end{itemize}
Although suboptimality and rank-adjustment 
have been individually considered in the literature in some
form~\cite{BMNonlinear2003,BVBNonconvex2018,DYC+OptimalStorage2020}, 
there is no discussion of how they can be combined into
a single SDP solver package, nor have the benefits of their combination been demonstrated.    We demonstrate the new
solver through comprehensive experiments using over 200 instances of
Maximum Cut, Minimum Bisection, Lov\'{a}sz Theta, and Cut Norm SDPs.
These experiments show that this simple combination is often faster
and more scalable than recently proposed Frank-Wolfe based methods for
computing moderately accurate solutions. We further demonstrate the applicability of \SDPLRp to data science problems,
including $\mu$-conductance, matrix completion, and $k$-means clustering.

In the remainder of the paper, we explain the details behind this approach. As noted above, many of these pieces have appeared individually in the literature before. However, they result in a \emph{simple} change to the SDPLR solver that has a remarkable impact on performance. 

\section{Preliminaries}

Matrices and vectors are written in bold. The norm \(\|\cdot\|\) refers to
the 2-norm of a vector and the spectral norm of a matrix, and
$\|\cdot\|_F$ denotes the Frobenius norm of a matrix. For a matrix
$\mX$, $\diag(\mX)$ denotes the vector consisting of its diagonal
entries.  For a vector $\vx$, $\Diag(\vx)$ denotes the diagonal matrix
whose diagonal entries are given by $\vx$.
    
Let \(G = (V, E)\) denote an undirected graph with vertex set \(V\)
and edge set \(E\).  We let \(uv\) denote the undirected edge between
vertex \(u\) and \(v\).  Each edge \(uv \in E\) is assigned a weight
\(w(uv)\). Non-edges have weight 0. The (weighted) degree of a vertex
$v$ is $\sum_{u: uv \in E}^{} w(uv)$, and let $\vd$ be the degree vector
where $\vd(u) = \deg(u)$. The degree matrix $\mD$ of a graph is simply
$\Diag(\vd)$. For any vertex set \(S\), let
\(\bar{S} = V \setminus S\) be its complement and the value of the cut induced
by \(S\) is
\(\textbf{cut}(S, \bar{S}) = \sum_{u \in S, v \in \bar{S}} w(uv)\).  The
volume of $S$, $\vol(S) \defeq \sum_{v \in S} \deg(v)$ is a size measure
for vertex sets, and $\vol(G) \defeq \vol(V) = 2 |E|$.
Let \(\mL\) be the Laplacian matrix of an undirected graph where
\(\mL_{i,j} = -w(ij)\) if \(i \neq j\) and
\(\mL_{i,i} = \smash{\sum_{j} w(ij)}\).

\section{Methodology and Algorithm}
In this section, we describe how we design \texttt{SDPLR+} to provide
stronger quality guarantees and make it more scalable than \texttt{SDPLR}
for \cref{eq:3}. We begin with a theoretical result about SDPs.

\subsection{Strong duality and a suboptimality bound}
\label{sec:suboptimality}
We design a suboptimality bound by using the fact that the trace bound
in \cref{eq:3} results in an unconstrained dual problem where strong
duality always holds.  In optimization, a suboptimality bound tells us
how far the current objective is away from the optimum, which gives a
precise measure of progress for the solver. In the context of
\cref{eq:4}, a suboptimality bound is any valid upper bound of
\(\langle \mC, \mX\rangle - \langle \mC, \mX^* \rangle\) where
\(\mX^*\) is the optimum of \cref{eq:4}.

A key tool for deriving a suboptimality bound is duality. Consider the
dual problem for \cref{eq:4}
\begin{align}
  \label{eq:5}
  \sup_{\vlambda} \; \vlambda^T \vb \;  \subjectto \; \mC - \cA^*(\vlambda) \succeq 0, \tag{SDD}
\end{align}
where \(\cA^*: \R^m \to \S^n\) is the adjoint of \(\cA\).  By
duality of convex programs, any feasible solution to \cref{eq:5} has
a dual value at most \(\langle \mC, \mX^*\rangle\).  Therefore the gap between
\(\langle \mC, \mX \rangle\) and the dual value serves as a suboptimality
bound. However, finding a non-trivial feasible solution to
\cref{eq:5} is not easy because of the positive semidefinite
constraint $\mC - \cA^*(\vlambda) \succeq 0$. Also, unless the SDP
satisfies strong duality, there may exist a gap between the primal and
dual solutions. 

The trace bound in \cref{eq:3}, however, changes the picture
entirely.  Consider the dual of \cref{eq:3}
\begin{align}
  \label{eq:12}
  \sup_{\vlambda} \; \vlambda^T \vb + \alpha \min\setof{\lambda_{\min} (\mC - \cA^*(\vlambda)), 0}, \tag{Trace-bounded SDD} 
\end{align}
we have the following result, which characterizes the relation between
\cref{eq:3} and \cref{eq:12}.
\begin{theorem}
 \label{thm:strong-duality} 
  When \cref{eq:3} is feasible, strong duality always holds between \cref{eq:3} and \cref{eq:12},
 and the optimum of \cref{eq:3} can be attained.
\end{theorem}
\begin{proof}
  Note that the space of $n \times n$ real matrices can be viewed as
  $\R^{n^{2}}$. We first show that $\Delta_{\alpha}$ is a compact set lying in
  this space. To show the compactness of $\Delta_{\alpha}$, we use Heine–Borel
  theorem to turn it into proving $\Delta_{\alpha}$ is closed and bounded.  The
  set $\Delta_{\alpha}$ is closed as it is the intersection of two closed sets:
  (a) the semidefinite cone $\{\mX \in \S^n: \mX \succeq 0 \}$ and (b) the set
  of trace-bounded matrices. The set $\Delta_{\alpha}$ is bounded because
  \begin{align*}
    \|\mX\|_F = \sqrt{\Tr(\mX\mX^{T})} = \sqrt{\Tr(\mX^2)} \le \sqrt{\lambda_{\max}(\mX) \Tr(\mX)} \le \sqrt{\Tr(\mX)^2} \le \alpha 
  \end{align*}      
  where we used a known inequality
  $\Tr(\mA \mB) \le \lambda_{\max}(\mA) \Tr(\mB)$ for positive semidefinite
  matrices (instances of this date back to at least~\citet{1098852}).

  Now let
  $f(\mX, \vlambda) = \langle\mC, \mX\rangle - \vlambda^{T}(\mathcal{A}(\mX) - \vb)$.  As
  \cref{eq:3} is feasible,
  \begin{align*}
  \cref{eq:3} = \min_{\mX \in \Delta_{\alpha}} \sup_{\vlambda} f(\mX,
    \vlambda).
  \end{align*}
  Because $\Delta_{\alpha}$ is compact and convex, and both
  $f(\cdot, \vlambda)$ and $f(\mX, \cdot)$ are linear, Sion's minimax
  theorem~\cite{Sion1958} says
  \begin{align*}
  \cref{eq:3} & =  
                \min_{\mX \in \Delta_{\alpha}} \sup_{\vlambda} f(\mX, \vlambda) \\
              & = 
  \sup_{\vlambda} \min_{\mX \in \Delta_{\alpha}} f(\mX, \vlambda) =
  \cref{eq:12}.  
  \end{align*}
  The optimum of \cref{eq:3} can be attained because
  $\Delta_{\alpha}$ is compact.
\end{proof}

Note that \cref{thm:strong-duality} can also be derived via extending the
original SDP into an equivalent new SDP with primal variable being an
$(n+1)\times(n+1)$ positive semidefinite matrix and encoding the trace
bound as a constraint in the new SDP. Then Theorem 2.2.5 of
\cite{helmberg2000semidefinite} implies strong duality because
the dual problem of the new SDP is strictly feasible. To help
understand why this result changes the picture of duality -- and closes
the duality gap for SDPs -- we present an illustrative example. 

\paragraph{Duality Gap Example} A classical example showing that a duality gap may
exist between \cref{eq:4} and \cref{eq:5} is
\begin{align*}
  \minimize &\quad  y_1 \\
  \subjectto &\quad
        \begin{bmatrix}
          0 & y_1 & 0 \\
          y_1 & y_2 & 0 \\
          0 & 0 & y_1 + 1 
        \end{bmatrix}
        \succeq 0.
\end{align*}
Because any principal submatrix of a symmetric positive semidefinite
matrix has to be positive semidefinite, we get $y_1^2 \leq 0$ and the
optimum of it is 0.  Writing it in the standard \cref{eq:4} form
gives us
\begin{align}
  \label{eq:16}
  \minimize_{\mX \succeq 0} & \quad \langle
         \Bigl[ \begin{smallmatrix}
           0 & \frac{1}{2} & 0 \\
           \frac{1}{2} & 0 & 0 \\
           0 & 0 & 0 
         \end{smallmatrix} \Bigr],
         \mX \rangle \\ 
  \subjectto & \quad \langle
        \Bigl[ \begin{smallmatrix}
           0 & 0 & 1 \\
           0 & 0 & 0 \\
           1 & 0 & 0
        \end{smallmatrix} \Bigr],
        \mX \rangle = 0,
  \langle
        \Bigl[ \begin{smallmatrix}
           0 & 0 & 0 \\
           0 & 0 & 1\\
           0 & 1 & 0
        \end{smallmatrix} \Bigr],
        \mX \rangle = 0 ,
 \nonumber  \\
  & \quad \langle
        \Bigl[ \begin{smallmatrix}
           1 & 0 & 0 \\
           0 & 0 & 0 \\
           0 & 0 & 0
        \end{smallmatrix} \Bigr],
        \mX \rangle = 0 ,
  \langle 
         \Bigl[ \begin{smallmatrix}
           0 & 1 & 0 \\
           1 & 0 & 0 \\
           0 & 0 & -2 
         \end{smallmatrix} \Bigr],
         \mX \rangle = -2, \nonumber 
\end{align}
whose dual is 
\begin{align}
 \label{eq:15} 
\maximize_{\vlambda \in \R^4}&\quad -2 \lambda_4 \\ 
\subjectto&\quad
\begin{bmatrix}
-\lambda_3 & \frac{1}{2} - \lambda_4 & -\lambda_1 \\
\frac{1}{2} - \lambda_4 & 0 & -\lambda_2 \\
-\lambda_1 & -\lambda_2 & 2\lambda_4
\end{bmatrix} \succeq 0 \nonumber.
\end{align}
Similarly, we have \((\frac{1}{2} - \lambda_{4})^2 \le 0\), which leads to
$\lambda_4 = \frac{1}{2}$ and \cref{eq:15} having optimum \(-1\). A duality
gap does exist.

However, if we add one trace bound \(\Tr(\mX) \le 1\) to \cref{eq:16},
its optimum stays unchanged and its dual becomes
\begin{align*}
\maximize_{\lambda \in \R^4} - 2 \lambda_4 + \min \left\{ 0, \lambda_{\min}
\left(\begin{bmatrix}
-\lambda_3 & \frac{1}{2} - \lambda_4 & -\lambda_1 \\
\frac{1}{2} - \lambda_4 & 0 & -\lambda_2 \\
-\lambda_1 & -\lambda_2 & 2\lambda_4
\end{bmatrix}\right)
\right\},
\end{align*}
whose optimum is also 0 (plug in
$\lambda_1 = \lambda_2 = \lambda_{3} = 0, \lambda_4 = - \frac{1}{2}$).

\begin{corollary}
Because \cref{eq:12} is unconstrained, any \(\vlambda\) is feasible
for this dual problem. As a result any dual solution lower bounds
\(\langle \mC, \mX^*\rangle\) and gives the following suboptimality bound
\begin{align}
\label{eq:14}
\langle \mC, \mX\rangle - \langle \mC, \mX^{*} \rangle \leq \langle\mC, \mX\rangle - \vlambda^T \vb - \alpha \min\{\lambda_{\min} (\mC - \cA^*(\vlambda)), 0\}. 
\end{align} 
\end{corollary}
The bound in this corollary is also present
in~\citet{DYC+OptimalStorage2020,angell2024fast}, although that bound
is presented under the assumption that strong duality holds. Our
results show this is not necessary.

Although this may look simple, it proves to be simpler and cleaner than
existing suboptimality bounds.  \Citet{BVBNonconvex2018}
give a suboptimality bound from the Riemannian optimization
perspective, which, however, requires a strong assumption that the
gradient of all constraints \(\cA(\mY \mY^T)\) are independent for
all \(\mY\) on the manifold
\(\cM = \{\mY \in \R^{n \times r}: \cA(\mY \mY^T) = \vb\}\).  \Citet{YTF+Scalable2021a} also addresses \cref{eq:3} but gives
a suboptimality bound by considering the Frank-Wolfe surrogate
duality gap for the augmented Lagrangian. Their suboptimality bound
has a more complicated form and by simple algebraic manipulation, one
can show that it shares a connection with \cref{eq:14} but is worse, which we illustrate in \cref{sec:suboptimality-other}.

\subsection{Implicit links to the trace-bounded dual in \texttt{SDPLR}}
The goal of this section is to explain that the augmented Lagrangian
strategy in \texttt{SDPLR} should do a reasonable job of producing
primal and dual iterates that are tight in terms of the suboptimality
bound~\cref{eq:14}. To do so, we need a brief recap of the augmented
Lagrangian method, before describing how the subproblems map into our
setting.

The augmented Lagrangian method is a popular framework for constrained
optimization problems (see~\citet{NWNumerical2006} for
example). Consider constrained problems of the following form
\begin{align}
\label{eq:9}
\minimize_{\vx \in \cD} \; f(\vx) \; \subjectto \; g(\vx) = \zeros \tag{Constrained Optimization}
\end{align}
where \(f: \cD \to \R\) is a real-valued objective function and
\(g: \cD \to \R^m\) encodes \(m\) constraints. From a minimax
perspective, \cref{eq:9} can be equivalently formulated as
\[ \min_{\vx \in \cD} \max_{\vlambda}f(\vx) - \vlambda^T g(\vx) \]
where
\(\vlambda\) is the Lagrangian multipliers. However, the inner
\(\max_{\vlambda}\) can be nonsmooth. The ALM framework tackles this
by introducing a quadratic proximal term penalizing \(\vlambda\)
deviating from the prior estimate \(\bar{\vlambda}\), which in turn
gives the following minimax problem
\begin{align}
\label{eq:10}
\min_{\vx \in \cD} \max_{\vlambda} f(\vx) - \vlambda^T g(\vx) -
\frac{1}{2\sigma}\|\vlambda - \bar{\vlambda}\|^2. 
\end{align}
We can see that the inner maximum of \cref{eq:10} is attained at
\begin{align}
\label{eq:8}
\vlambda = \bar{\vlambda} - \sigma g(\vx) 
\end{align}
and plugging it into \cref{eq:10} simplifies the outer minimization
problem to
\begin{align}
 \label{eq:11} 
\minimize_{\vx \in \cD} f(\vx) - \bar{\vlambda}^T g(\vx) + \frac{\sigma}{2}\|g(\vx)\|^2. 
\end{align}
A standard ALM implementation involves solving the inner and outer
optimization problems of \cref{eq:10} alternately and \cref{eq:8},
\cref{eq:11} are usually referred to as dual and primal updates.
Usually $\sigma$ will increase across iterations to sharpen the proximal
term, which is quite intuitive as \cref{eq:10} is equivalent to
\cref{eq:9} when $\sigma\rightarrow\infty$.

We use the ALM method to provide good primal and dual estimates for
\cref{eq:3} and \cref{eq:12} in \texttt{SDPLR+}, just as it was used in
\texttt{SDPLR}. In particular, the primal update is
\begin{align*}
\minimize_{\mX \in \Delta_{\alpha}} \; \langle \mC, \mX\rangle - \bar{\vlambda}^T \left( \mathcal{A}(\mX) - \vb \right)
+ \frac{\sigma}{2} \|\mathcal{A}(\mX) - \vb\|^2,
\end{align*}
and the dual update is
$\vlambda = \bar{\vlambda} - \sigma(\mathcal{A}(\mX) - \vb)$. Using the
Burer-Monteiro factorization over the factors $\mY\mY^T$ of $\mX$
turns the primal and dual updates into
\begin{align*}
\minimize_{\mY \in \R^{n \times r} : \Tr(\mY\mY^T) \leq \alpha} \; \langle \mC, \mY \mY^T\rangle - \bar{\vlambda}^T \left( \mathcal{A}(\mY \mY^T) - \vb \right)
+ \frac{\sigma}{2} \|\mathcal{A}(\mY \mY^T) - \vb\|^2
\end{align*}
and
$\vlambda = \bar{\vlambda} - \sigma(\mathcal{A}(\mY \mY^T) - \vb)$. There is a strong
parallelism between these updates and the primal, dual updates of
\texttt{SDPLR}.  The only difference is the constraint
$\Tr(\mY\mY^T) \leq \alpha$, which is usually not violated badly because the
trace bound encoded in the constraints implicitly prevents it from
deviating too much.  Intuitively, this shows that although originally
designed to provide estimates for \cref{eq:4} and \cref{eq:5},
\texttt{SDPLR} is implicitly optimizing the primal and dual of the
trace-bounded version. This shows why this particular bound is likely
to be successful.

\subsection{The overall \texttt{SDPLR+} algorithm and dynamic rank updates}
\label{sec:sdplr-overall}
We summarize key points of the overall approach and then give the core
pseudocode. Like \texttt{SDPLR}, we use the augmented Lagrangian
approach. The minimization problem in the primal update is
approximately solved using L-BFGS~\cite{liu1989limited}.  The line
search in L-BFGS is performed by solving a cubic
equation~\cite{BCComputational2006}.


\paragraph{Eigenvalue Computation}
One critical computation involved in evaluating the suboptimality
bound is evaluating the minimum eigenvalue of
\(\mC - \cA^*(\vlambda)\). We use the Lanczos method with random
start \cite{KWEstimating1992} and adopt the memory-efficient
implementation from \citet{YTF+Scalable2021a}. In our case, we 
compute the minimum eigenvalue when evaluating the suboptimality to
evaluate termination criteria. Thus, we run the method longer
as the outer ALM iterations progress to achieve a higher accuracy.
In particular, similar to \citet{YTF+Scalable2021a}, we take
$2 \sqrt{t} \log n$ steps, where $t$ is the number of ALM
iterations as in \cref{alg:sdplr+}.

\paragraph{Dynamic Rank Update}
Although Burer-Monteiro based methods have good theoretical guarantees
when \(r = \Omega(\sqrt{m})\), a decision variable \(\mY\) with
\(\Omega(n \sqrt{m})\) non-zero entries is still far from scalable. Hence
in practice, a much smaller rank is usually picked to speed up the
computation and save memory.  The downside is that
Burer-Monteiro based methods may fail to converge to the optimum.  A
common strategy is to increase the rank dynamically. \texttt{SDPLR}
and \texttt{Manopt+}~\cite{BMNonlinear2003,BVBNonconvex2018} both
dynamically increase the rank from a small number to
\(\Omega(\sqrt{m})\), but they either do not consider early termination or
implement early termination based on the heuristic that the answer is
close to the optimum if the change of objective value is small across
two consecutive iterations.  Now because we have an easy-to-compute
suboptimality bound, we can terminate the dynamic rank update once the
desired suboptimality precision is achieved.

\paragraph{The \texttt{SDPLR+} algorithm}
We state the core algorithm of \texttt{SDPLR+}.  We
use $\eta_t, \omega_t, \xi_t$ to denote the stationarity tolerance, primal
infeasibility tolerance, and suboptimality tolerance in the $t$-th
iteration, respectively. We let $\omega^{*}, \xi^{*}$ denote the desired
primal infeasibility and suboptimality tolerances input by the user.
We use $\vlambda$ to denote the Lagrangian multipliers estimated and
$\sigma$ to denote the smoothing parameter in the augmented Lagrangian method.
The augmented Lagrangian function is
\begin{align}
\label{eq:13}
f(\mY) = \langle\mC, \mY\mY^T \rangle - \vlambda^{T}(\mathcal{A}(\mX) - \vb) + \frac{\sigma}{2} \norm{\mathcal{A}(\mX) - \vb}^{2}.\nonumber
\end{align}
For ease of notations, we overload the notations and let
\begin{align*}
&\eta(\mY, \vlambda) = \frac{\norm[F]{\nabla f(\mY)}}{1 + \norm[F]{\mC}}, \;
\omega(\mY)
= \frac{\norm{\cA(\mY \mY^T) - \vb}}{1 + \norm{\vb}}, \; \\
& \xi(\mY, \vlambda) 
= \frac{\langle\mC, \mY\mY^T \rangle - \vlambda^T \vb - \alpha \{\lambda_{\min} (\mC - \cA^*(\vlambda)), 0\}}{1 +| \langle\mC, \mY\mY^{T} \rangle | }
\end{align*}
denote the relative stationarity, primal infeasibility and suboptimality
evaluated at the current iterates $\mY, \vlambda$ respectively.
\begin{algorithm}[t]
\begin{algorithmic}[1]
\footnotesize 
\Require problem data \(\mC, \mathcal{A}, \vb\), trace bound $\alpha$, output precisions $\omega^{*}, \xi^{*}$ and initial rank $r$.      
  \Ensure $\mY$ satisfying desired primal infeasibility and suboptimality precision
\State $\sigma_0 = 2, \eta_0 = 1/\sigma_0, \omega_0 = 1/\sigma_0^{0.1}, \gamma = 4$
\State Randomly initialize $\mY$ using initial rank $r$ and $\vlambda$.
  \For {$t = 0, 1, \ldots$}
\State Use L-BFGS to find a new $\mY$ such that $\eta(\mY, \vlambda) \leq \eta_t$ \Comment{Find a stationary point}
\If{$\omega(\mY) \leq \omega_t$} \Comment{If primal infeasibility is small}
\If{$\omega(\mY) \leq \omega^{*}$} 
\IfThenElse{$\xi(\mY, \vlambda) \leq \xi^{*}$}{Break \Comment{Find a solution}}{$\gamma = \gamma - 1$ \Comment{Decrease dynamic rank update counter}}
\EndIf
\State $\vlambda = \vlambda - \sigma(\mathcal{A}(\mY\mY^T) - \vb)$ \Comment{Update the dual estimate}
\State $\eta_{t+1} = \eta_t/\sigma_t, \omega_{t+1} = \omega_t/\sigma_t^{0.9}$
\Else \State $\sigma_{t+1} = 2\sigma_{t}, \eta_{t+1} = 1/\sigma_{t+1}, \omega_{t+1} = 1/\sigma_{t+1}^{0.1}$ \Comment{Enlarge primal infeasibility penalty}
\EndIf
\If{$\gamma = 0$} \Comment{Double the rank}
\State $r = \min\{2r, \lfloor \sqrt{2m} + 1\rfloor\}$, regenerate $\mY$ and $\vlambda$ accordingly,
set $\gamma = 4$.
  \EndIf
\EndFor
\end{algorithmic}
\caption{Pseudocode of the core algorithm of \texttt{SDPLR+}}
\label{alg:sdplr+}
\end{algorithm}
We generate the initial $\mY$ and $\vlambda$ randomly.
The pseudocode
is summarized in \cref{alg:sdplr+}. The main difference with
\texttt{SDPLR} is that we apply the two techniques we introduced.
Also, we customized the augmented Lagrangian update rules, adopting
one from \citet{NWNumerical2006}.



\section{Related Work}
For detailed discussions on solving semidefinite programs, see the
surveys \citet{MHASurvey2019,YYSURVEY2015}, and for discussion of
low-rank solutions, see~\citet{LSYLowRank2016}. This section discusses
the most relevant related work.

\subsection{Scalable SDP solvers}
The two most related classes of scalable SDP solvers are
Burer-Monteiro based methods and Frank-Wolfe based methods. As
mentioned before, the core idea of Burer-Monteiro based methods is to
optimize over the low-rank factors of the decision matrix.  Besides
\texttt{SDPLR}, another line of Burer-Monteiro methods is Riemannian
optimization based, which optimizes the low-rank factors on the
Riemannian manifolds with or without extra constraints, instead of in
the Euclidean space~\cite{BVBNonconvex2018,LBSimple2019}.
Concurrently to our research, \citet{monteiro2024low}
proposed an efficient augmented Lagrangian method for \cref{eq:3}
based on combining an adaptive inexact proximal point method with
inner acceleration and Frank-Wolfe steps.

Roughly, the Frank-Wolfe-based methods can be categorized into two
classes: (1) \citet{YFCConditionalGradientBased2019} propose a solver
\texttt{CGAL} that tackles the constrained problem \cref{eq:4} via the
augmented Lagrangian method where the primal update is performed by
one Frank-Wolfe update instead of the conventional minimization.
Furthermore, they propose a more scalable solver which uses efficient
sketching techniques for symmetric positive semidefinite matrices to
cut the memory usage~\cite{TYUCFixedrank2017,YTF+Scalable2021a}.  (2)
Another line is similar to the penalty method that integrates the
constraints into the objective with properly chosen penalty function
and coefficient, and optimizes the objective efficiently using
Frank-Wolfe method~\cite{HazSparse2008}. In this vein,
\citet{SNSMemoryEfficient2021a,SNSMemoryEfficient2021} recently
observed that for SDPs that are designed for combinatorial problems
and have a hyperplane rounding algorithm, one can cut the memory cost
by only tracking one random vector with a distribution
$\mathcal{N}(\zeros, \mX)$. In this way, $\mX$ is implicitly encoded and is not
required to be stored.  Recently, \citet{PGSScalable2023} combined
this extreme point sampling idea with a reformulation from
\citet{NesBarrier2011a} to design a scalable SDP solver \texttt{SCAMS}
for Max Cut.

Concurrent to our work, \citet{angell2024fast} designed \texttt{USBS}
which adapts spectral bundle and combines it with
sketching. This demonstrates a scalable solver
in tests with $n$ in the millions.

\subsection{Theoretical Guarantees of Burer-Monteiro Methods}

Due to the great empirical success of the Burer-Monteiro method, there has
been an extensive effort for developing theory for it.  Recently
\citet{BVBNonconvex2018,BVBDeterministic2020} showed that if
the feasible set is a smooth manifold, the cost matrix $\mC$ is generic,
and $r$ satisfies $r(r+1)/2 > m$, then equality-constrained SDPs have
no spurious 2nd-order critical points. Put another way, all 2nd-order
critical points are global optima under these
conditions.  \Citet{CifBurer2021} generalized their result to broader
classes of SDPs.  On the other hand, \citet{WWRank2020} show that the
Barvinok-Pataki bound is essential as the 2nd-order critical points of
Burer-Monteiro are not generically optimal even if the SDP admits a
global optimal solution with rank lower than the Barvinok-Pataki
bound. \Citet{OSVBurerMonteiro2022} show that the assumption that the
cost matrix is generic is necessary and construct examples where
Burer-Monteiro has spurious 2nd-order critical points.  Finally, in
the setting of smooth analysis, \citet{CMPolynomial2022} show that for
$r$ larger than the Barvinok-Pataki bound, the Burer-Monteiro method
can provide an SDP solution to any accuracy in polynomial time.

\section{Experiments}

\label{sec:experiments} 
We evaluate our algorithm and a variety of other solvers on over 200
problems in a test set described in the following subsection. We test a
slightly different set of solvers for each problem because some solvers
target more specific subclasses of \cref{eq:3}.
While we
would have liked to test all of the recently proposed solvers, we were
unable to find code available for~\citet{monteiro2024low}. We do more
limited comparisons against the \texttt{USBS} method 
from~\citet{angell2024fast} due to issues encountered when running
their code. The results in \cref{sec:extra-solvers} 
show that we are likely faster in many
scenarios.  

\subsection{Experimental setup}

Since the remaining solvers are written in many different
programming languages, we evaluate all solvers in a single threaded
scenario with many single threaded solvers running concurrently. This
should be the most energy efficient and representative scenario given
the dynamic frequency scaling of modern CPUs. We ran our evaluation on
different servers with slightly different processors. In order to
compare across problems, we guarantee that all solvers on a given
problem instance were run on the same server. We also limit total RAM
access for each solver to 16GB and limit each solver to a runtime of 8
hours.  On all problem instances, we start the dynamic rank update of
\texttt{SDPLR+} from 10, and set the L-BFGS history size to 4.

We ran our experiments on three servers, two of which had the same
configuration.
\begin{itemize}
\item One server with 64 cores and each core has two threads. The CPU
model is Intel(R) Xeon(R) CPU E7-8867 v3 @ 2.50GHz.
\item Two servers with 28 cores each and each core has two
threads. The CPU model is Intel(R) Xeon(R) CPU E5-2690 v4 @ 2.60GHz.
\end{itemize}
In our experiments, we let each solver solve each problem instance
using only one thread.  We use \texttt{GNU
Parallel}~\cite{gnu-parallel} to parallelize the benchmarking and
timeout each solver after running for 8 hours by setting
\texttt{-{}-timeout 28800}.  We limit the total RAM access to 16GB via
\texttt{ulimit -d 16777216}.

Throughout all experiments, we stop
all the optimizers according to the following quality metrics: \\
\noindent\begin{minipage}[b]{.5\linewidth}
\begin{align}
\label{eq:18}
\frac{\| \mathcal{A}(\mX) - \vb\|}{1 + \|\vb\|} \tag{Primal Infeasibility}
\end{align}
\end{minipage}%
\begin{minipage}[b]{.5\linewidth}
\begin{align}
\label{eq:19}
\frac{\langle\mC, \mX\rangle - \langle\mC, \mX^{*}\rangle}{1 + | \langle \mC, \mX \rangle |} \tag{Suboptimality}
\end{align}
\end{minipage}

We study the solvers on four distinct problems: Max Cut, Minimum
Bisection, Lov\'{a}sz Theta, and Cut Norm. We only evaluated a solver
if the authors provided an implementation for solving SDP or if it
could be implemented in a way where we were confident would accurately
represent performance.

\subsection{Datasets}
The SDPs we use correspond to graph problems. As such, there is a
wealth of possible data for benchmarking the solvers. We chose these
graphs based on other comparisons of SDP solvers.
\begin{itemize}
\item \gset~\cite{Helmberg2000}, accessed via~\cite{YYY-Gset}: This is a set of 71
randomly generated square matrices from the \texttt{rudy} graph generator~\cite{Rinaldi-Rudy} with either binary or
\(-1, 0, +1\) values.  Sizes range from 800 to 20000 rows and up to
80000 non-zero entries. 
\item \snap~Datasets~\cite{snapnets}: We pick a list of 4 web
graphs and 11 community graphs from \snap 
\begin{itemize} 
\item web-BerkStan, web-Google, web-Stanford~\cite{LLDMCommunity2008}
\item web-NotreDame~\cite{AJBDiameter1999}
\item ca-AstroPh, ca-CondMat, ca-GrQc, ca-HepPh, ca-HepTh \\ \cite{LKFGraph2007}, 
\item com-youtube, com-dblp, com-amazon~\cite{YLDefining2015},
\item email-Enron~\cite{KYIntroducing2004},
\item musae-facebook~\cite{RASMultiScale2021},
\item feather-deezer-social~\cite{RSCharacteristic2020}.
  \end{itemize} 
\item \dimacsten ~\cite{DBLP:conf/dimacs/2012}: This is a set of 151
different kinds of graphs from the 10th \dimacs Implementation
Challenge for graph partitioning and graph clustering subtasks.  For
\dimacsten, for the interest of benchmarking time, we sort graphs by
the number of vertices they contain and keep the first 116
graphs. The largest one, \texttt{rgg\_n\_2\_20\_s0}, has
$n = \text{1,048,576}$ and $\text{89,239,674}$ nonzeros. Data from
the following papers is included: \citet{Baird_1989,DBLP:conf/ipps/HoltgreweSS10,DBLP:conf/wea/SafroSS12,DBLP:journals/cluster/ChanLA12,DBLP:journals/jgo/SoperWC04,Murphy-2010-graph500,Watts_1998,geisberger2008better,skitter_router_adjacencies,kluge2011efficient}.
\end{itemize}

\begin{figure}[t]
\centering
\hfill
\subfloat[Stats of $n$]{\includegraphics[width=0.32\linewidth]{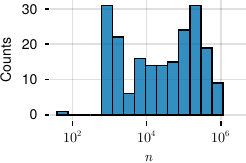}}%
\hfill
\subfloat[Stats of Non-zeros]{\includegraphics[width=0.32\linewidth]{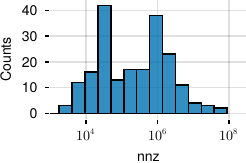}}%
\hfill
\label{fig:dataset-details}
\caption{Statistics of problem dimension, $n$, and the number of non-zeros in the 202 matrices in our test set. }  
\end{figure}

In total, we build a collection of 202 graphs; statistics are summarized in
\cref{fig:dataset-details}.  For all the graphs, we drop directions
and symmetrize the adjacency matrix if they are directed.
For \snap datasets, we take the
largest components of the graphs and remove self-loops.

\subsection{Performance Profile Plots} 
We use performance profile
plots~\cite{dolan2002benchmarking,BenchmarkProfiles} to compare the
solvers across this range of problems. Since each solver \emph{may not
be the fastest} or \emph{best} on every problem instance, this shows
a profile of how often the solvers are within a $\tau$-\emph{bound} of
the best. It's reminiscent of a receiver-operator curve or
precision-recall curve. The goal is to be ``up and to the left'' in
this case. This means that a given solver was the best or within a
small factor of being the best on the largest number of problems. For
all performance profile plots throughout this paper, the x-axis $\tau$
represents the ratio between the current solver and the best one with
regard to a specific metric (time or discrete objective value).  The
point on the right of a solver curve corresponds to what fraction of
the total set of problems it solved within the time bound.

\subsection{Max Cut}
\label{sec:exp-max-cut}

%

\begin{figure}[t]
\captionsetup[subfigure]{justification=centering}
\centering
\subfloat[Runtime by Problem Size][Runtime by \\Problem Size]{\label{fig:max-cut-scalability}\includegraphics[width=0.28\linewidth]{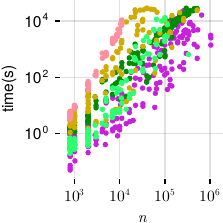}}%
\hfill
\subfloat[Performance Profile of Runtime][Performance Profile \\ of Runtime]{\label{fig:max-cut-running-time}\includegraphics[width=0.28\linewidth]{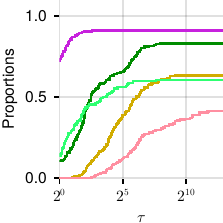}}%
\hfill
\subfloat[Performance Profile of Rounded Cuts][Performance Profile \\ of Rounded Cuts]{\label{fig:max-cut-cuts}\includegraphics[width=0.28\linewidth]{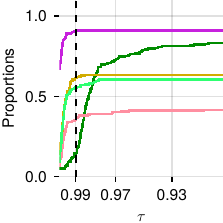}}%
\hfill
\subfloat{\includegraphics[width=0.15\linewidth]{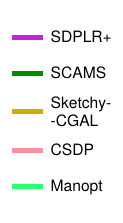}}%
\caption{Results on Max Cut. We set tolerance
$\varepsilon = 10^{-2}$, and test each solver on all 202 graphs.  We give each solver 8 hours to run before ending the experiment with a failure case.  \cref{fig:max-cut-scalability}
and \cref{fig:max-cut-running-time} display how the running time
of each solver scales with the problem size $n$ and the running
time performance profile plot of all the solvers respectively.
  Note that \texttt{SDPLR+} is the fastest solver for most of the
problem instances and is within a small factor of the optimal
solver for the other instances. While \texttt{SCAMS} also
exhibits good scalability, from \cref{fig:max-cut-cuts}, we see
that it usually produces smaller cuts.  }
\label{fig:max-cut}
\end{figure}

Max Cut is a combinatorial problem that aims to find the maximum
cut induced by any vertex set \(S\), i.e., solving
\(\max_{S \subseteq V} \textbf{cut}(S, \bar{S})\).  Because Max Cut is NP-hard
to solve, a common strategy is relaxing it to the following
semidefinite program and then rounding the real solution to a discrete
partition \((S, \bar{S})\)~\cite{GWImproved1995},
\begin{align}
\label{eq:1}
\tag{Max Cut}
\maximize_{\mX \succeq 0} \; \tfrac{1}{4} \langle \mL, \mX \rangle \; \subjectto \; \diag(\mX) = \ones.  
\end{align}

On Max Cut, we compare our method against the following solvers:
\begin{enumerate}
\item \texttt{CSDP}~\cite{BorCSDP1999}: A C library for solving
\cref{eq:4} using predictor and corrector algorithm. \texttt{CSDP}
provides both primal and dual solutions.
\item \texttt{SketchyCGAL}~\cite{YTF+Scalable2021a}: A low-rank
variant of CGAL method which sketches the primal variable of CGAL
using Nystr\"om Sketching, which has great scalability for producing
moderately accurate solutions. We set the sketching size
$R = 10$ throughout all experiments.
\item \texttt{SCAMS}~\cite{PGSScalable2023}: A memory-efficient
Frank-Wolfe based solver designed specifically for \cref{eq:1}.
\item Riemannian Optimization Methods: We adopt one efficient
Riemannian optimization implementations designed for \cref{eq:1}
in~\citet{BVBNonconvex2018}, \texttt{Manopt}. They turn the
diagonal-constrained \cref{eq:1} into an unconstrained optimization
problem on the Oblique manifold
\(\text{OB}(r, n) = \{\mY \in \R^{n \times r} : \diag(\mY \mY^T) =
\ones\}\) and apply Riemannian Trust Region
method~\cite{absil2007trust}.  \texttt{Manopt} sets
\(r = \lceil \frac{\sqrt{8n+1}}{2} \rceil\) which satisfies the
Barvinok-Pataki bound.
\end{enumerate}

Because \texttt{SCAMS} requires positive degrees, we preprocess all
graphs to take the absolute value of edge weights to make sure all
edge weights are positive for Max Cut.

For a given tolerance parameter $\varepsilon$, we stop \texttt{SDPLR+},
\texttt{SketchyCGAL} and \texttt{Manopt} when their \cref{eq:18} and
\cref{eq:19} are both smaller than $\varepsilon$. \texttt{Manopt} originally
terminates based on the norm of the Riemannian gradient, we exploited
the suboptimality bound for Max Cut introduced in
\citet{BVBNonconvex2018} to let it stop based on the suboptimality. In
particular, we exponentially decay the Riemannian gradient norm
tolerance from an initial value of 10 by a factor of 5 until the
suboptimality reaches the desired accuracy $\varepsilon$. Because \texttt{SCAMS}
is solving a reformulation whose optimum is the square root of the
optimum of \cref{eq:1}, we set its tolerance as $\varepsilon/2$. As
\texttt{CSDP} is tackling the following slightly different dual
problem
\begin{align}
\label{eq:25}
\maximize_{\vlambda, \mZ \succeq 0} \; \vlambda^T \vb \; \subjectto \; \mC - \cA^*(\vlambda) = \mZ,
\end{align}
it has two more metrics besides \cref{eq:18}, dual infeasibility and relative duality gap, which are 
\begin{align}
\label{eq:28}
\frac{\| \mathcal{A}^{*}(\vlambda) - \mC - \mZ\|_{F}}{1 + \|\mC\|} \quad \text{and} \quad
 \frac{\langle\mX, \mZ\rangle}{1 + |\langle\mC, \mX\rangle| + |\vlambda^T \vb|}
\end{align}
respectively. We set their tolerance all to $\varepsilon$.  We pick the trace
bound $\alpha$ as $n$ for \texttt{SDPLR+} and \texttt{SketchyCGAL}.
Furthermore, as Max Cut has a nice hyperplane rounding algorithm to
extract discrete solutions from the SDP solutions, we compare the
quality of the rounded cuts induced by SDP solutions of different
solvers. The results are summarized in \cref{fig:max-cut}. Overall,
\texttt{SDPLR+} shows better scalability and solution quality.

\subsection{Minimum Bisection}
\label{sec:exp-minimum-bisection}
Minimum Bisection~\cite{GAREY1976237}, which minimizes the bisection width,
has numerous applications~\cite{simon1991partitioning,sanchis1989multiple}.
Formally, it aims at dividing the vertex set \(V\) into two sets
\(S, \bar{S}\) with equal size, i.e., \(|S| = |\bar{S}|\) and
minimizing \(\textbf{cut}(S, \bar{S})\). The corresponding SDP
relaxation is
\begin{align}
\label{eq:2}
\tag{Minimum Bisection}
\minimize_{\mX \succeq 0} \; \tfrac14 \langle \mL, \mX \rangle \; \subjectto \; \diag(\mX) = \ones, \ones^T \mX \ones = 0.
\end{align}
Because the original minimum bisection aims to divide the graph into
two equal-size pieces, we add one dummy isolated vertex if needed. In
addition to comparing the generic SDP solvers \texttt{SketchyCGAL} and
\texttt{CSDP}, we compare with three simple constrained Riemannian
optimization solvers implemented in~\citet{LBSimple2019}, which are
the Riemannian Augmented Lagrangian method (\texttt{RALM}), exact penalty
method with smoothing via the log-sum-exp function (\texttt{\Qlse}),
and via a pseudo-Huber loss and a linear-quadratic loss
(\texttt{\Qlqh}).  Because \texttt{RALM}, \texttt{\Qlse},
\texttt{\Qlqh} have quite complicated termination conditions, we keep
them unchanged.  We pick the trace bound $\alpha$ as $n$ for
\texttt{SDPLR+} and \texttt{SketchyCGAL}.  We also apply hyperplane
rounding to extract bisections from SDP solutions. We regard solutions
provided by \texttt{RALM}, \texttt{\Qlse}, \texttt{\Qlqh} with
objective $\frac{1}{4} \langle\mL, \mX\rangle$ larger than
$(1 + \varepsilon)$ times the size of the best minimum bisection extracted as
underoptimized and discard them. The results are summarized in
\cref{fig:minimum-bisection}. Again, we see a clear advantage to
\texttt{SDPLR+} in speed and accuracy.


\begin{figure}[t]
\captionsetup[subfigure]{justification=centering}
\centering
\subfloat[Runtime by Problem Size][Runtime by \\ Problem Size]{\label{fig:minimum-bisection-scalability}\includegraphics[width=0.28\linewidth]{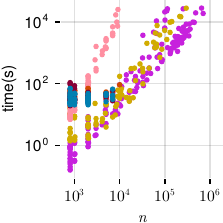}}%
\hfill
\subfloat[Performance Profile of Runtime][Performance Profile\\ of Runtime]{\label{fig:minimum-bisection-dt}\includegraphics[width=0.28\linewidth]{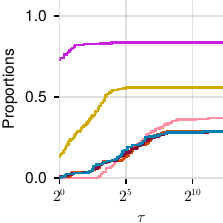}}%
\hfill
\subfloat[Performance Profile of Rounded Cuts][Performance Profile \\ of Rounded Cuts]{\label{fig:minimum-bisection-cut}\includegraphics[width=0.28\linewidth]{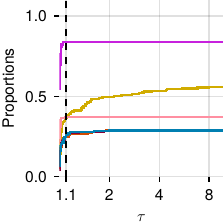}}%
\hfill
\subfloat{\includegraphics[width=0.15\linewidth]{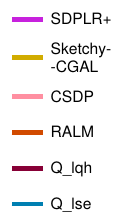}}%
\caption{Results on Minimum Bisection. We set tolerance
$\varepsilon = 10^{-2}$, and test each solver on all 202 graphs.  We time
out each solver after 8 hours.
  \cref{fig:minimum-bisection-scalability}
and \cref{fig:minimum-bisection-dt} display how the running time
of each solver scales with the problem size $n$ and the running
time performance profile plot of all the solvers respectively.
  Note that \texttt{SDPLR+} is overall the fastest solver. The
constrained Riemannian optimization methods fail on over half of
the problems. From \cref{fig:minimum-bisection-cut}, we
see \texttt{SDPLR+} often provides a better rounded minimum
bisection than \texttt{SketchyCGAL}.}
  \label{fig:minimum-bisection}
\end{figure}

\subsection{Lov\'{a}sz Theta}
The Lov\'{a}sz theta function~\cite{lovasz1979shannon} of an
undirected graph $G = (V, E)$, $\vartheta(G)$, can be computed by
\begin{align}
\label{eq:17}
\tag{Lov\'{a}sz Theta}
\vartheta(G) = \maximize_{\mX \succeq 0} \quad & \ones^T \mX \ones \\
\subjectto \quad & \Tr(\mX) = 1 \nonumber\\
& \mX_{ij} = 0\ \text{for all } ij \in E. \nonumber
\end{align}
Computing $\vartheta(G)$ reveals key characteristics of a graph
because of the sandwich result
$\alpha(G) \leq \vartheta(G) \leq \chi(\bar{G})$ where $\alpha(G)$ is the size of the
maximum independent set in $G$ and $\chi(\bar{G})$ is the chromatic
number of the complement of $G$.  We pick the trace bound $\alpha$ as
$1$ for \texttt{SDPLR+} and \texttt{SketchyCGAL}.
\cref{fig:lovasz-theta-cut-norm} summarizes results and shows
that \texttt{SketchyCGAL} is often faster than \texttt{SDPLR+} for
this SDP.

\subsection{Cut Norm}
The cut norm is relevant to a variety of graph and matrix
problems~\cite{alon1994algorithmic,frieze1999quick}. For any
$\mA \in \R^{m \times n}$, its cut norm is defined as
\begin{align*}
\|\mA\|_{\text{cut}} = \max_{S \subseteq [m], T \subseteq [n]} \sizeof{ \sum_{i \in S} \sum_{j \in T} 
\mA_{ij}}.
\end{align*}
While computing the cut norm is not tractable, Alon and
Naor gave an approximation algorithm based on solving the
following SDP and rounding its solution~\cite{ANApproximating2004}
\begin{align}
\label{eq:6}
\tag{Cut Norm}
\maximize_{\mX} \; \tfrac{1}{2}\langle
\bigl[ \begin{smallmatrix}
\zeros & \mA \\
\mA^T & \zeros
\end{smallmatrix} \bigr],
\mX \rangle \; 
\subjectto \; \diag(\mX) = \ones.
\end{align}
We compare \texttt{SDPLR+} against \texttt{SketchyCGAL} and
\texttt{CSDP}.  We pick the trace bound $\alpha$ as $m+n$ for
\texttt{SDPLR+} and \texttt{SketchyCGAL}.  The results are summarized
in \cref{fig:lovasz-theta-cut-norm}, which show that
\texttt{SDPLR+} is often much faster than \texttt{SketchyCGAL} for
this SDP.


\begin{figure}[t]
\captionsetup[subfigure]{justification=centering}
\centering
\subfloat[Lov\'{a}sz Theta \\ Runtime by Prob.~Size][Lov\'{a}sz Theta \\ Runtime by \\ Problem Size]{\includegraphics[width=0.21\linewidth]{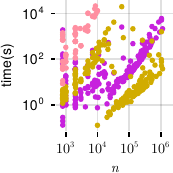}}%
\hfill
\subfloat[Lov\'{a}sz Theta \\Runtime Perf.~Prof.][Lov\'{a}sz Theta \\Runtime Perf. Prof.]{\includegraphics[width=0.21\linewidth]{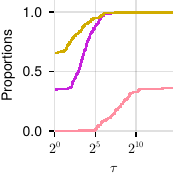}}%
\hfill
\subfloat[Cut Norm \\ Runtime by Prob.~Size][Cut Norm \\ Runtime by \\ Problem Size]{\includegraphics[width=0.21\linewidth]{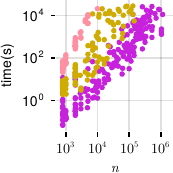}}%
\hfill
\subfloat[Cut Norm \\ Runtime Perf.~Prof.][Cut Norm \\ Runtime Perf. Prof.]{\includegraphics[width=0.21\linewidth]{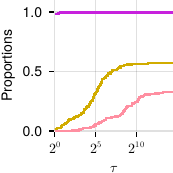}}%
\hfill
\subfloat{\includegraphics[width=0.12\linewidth]{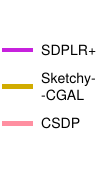}}
\caption{Results on Lov\'{a}sz Theta and Cut Norm.}
  \label{fig:lovasz-theta-cut-norm} 
\end{figure}

\subsection{Ablation studies}
We perform three ablation
studies. \Cref{sec:sdplr-vs-sdplr-plus} shows how the changes
in \texttt{SDPLR+} compare to \texttt{SDPLR} in runtime. As would be
expected, \texttt{SDPLR+} is faster on the vast majority of
instances. \Cref{sec:fixed-rank-study} studies a number of
solvers with a fixed \emph{rank} or fixed \emph{sketch size}. This
eliminates the dynamic behavior in some of the solvers.

\subsubsection{SDPLR vs. SDPLR+}
\label{sec:sdplr-vs-sdplr-plus}

We perform some simple comparison experiments to demonstrate
\texttt{SDPLR+} is faster than the original \texttt{SDPLR}. We let
them both solve Max Cut on \gset, which contains 71 small and moderate
graphs. We terminate \texttt{SDPLR+} once both \cref{eq:18} and
\cref{eq:19} are smaller than $\varepsilon$ and \texttt{SDPLR} once \cref{eq:18}
is smaller than $\varepsilon$. We slightly modified \texttt{SDPLR} to compute
the \cref{eq:18} in a Euclidean scaling for fair comparison. The
results are summarized in \cref{fig:sdplr-vs-sdplrplus}. As we can
see, \texttt{SDPLR+} is faster on the vast majority of instances.

\begin{figure}[htb]
\centering
\hfill
\subfloat[Runtime by Problem Size]{\includegraphics[width=0.32\linewidth]{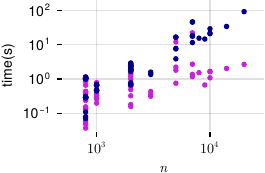}}%
\hfill
\subfloat[Runtime Perf. Prof.]{\includegraphics[width=0.32\linewidth]{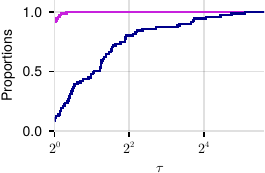}}%
\hfill
\subfloat{\includegraphics[width=0.18\linewidth]{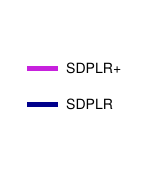}}%
\hfill
\caption{Comparing \texttt{SDPLR+} against \texttt{SDPLR} for
solving Max Cut on \gset graphs.}
  \label{fig:sdplr-vs-sdplrplus}
\end{figure}

\subsubsection{Higher Precision}
In practical applications, speed and scalability naturally trade off against high precision.
Here we study how fast and scalable
\texttt{SDPLR+} is when the precision requirement is higher; in particular
we tighten both the primal infeasibility and suboptimality tolerances
to $10^{-4}$ and $10^{-6}$. \Cref{fig:high-precision-max-cut},
\Cref{fig:high-precision--minimum-bisection} and
\Cref{fig:high-precision-lovasz-theta-cut-norm} summarize the results.

We observe that \texttt{SDPLR+} is still the fastest and most scalable
solver on Minimum Bisection and Cut Norm. On Max Cut, \texttt{SDPLR+}
is quite competitive when $\varepsilon=10^{-4}$ but \texttt{Manopt} has better
performance when $\varepsilon = 10^{-6}$. \texttt{SDPLR+} struggles on finding
high-precision solutions on Lov\'{a}sz Theta. Also, we see that
\texttt{CSDP} has relatively stable performance across different
precisions and it mostly fails on instances requiring too much memory.

\begin{figure}[p]
\captionsetup[subfigure]{justification=centering}
\centering
\subfloat[Runtime by Problem Size][Runtime by \\Problem Size]{\label{fig:max-cut-scalability-tol-0.0001}\includegraphics[width=0.28\linewidth]{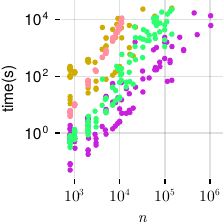}}%
\hfill
\subfloat[Performance Profile of Runtime][Performance Profile \\ of Runtime]{\label{fig:max-cut-running-time-tol-0.0001}\includegraphics[width=0.28\linewidth]{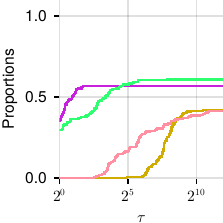}}%
\hfill
\subfloat[Performance Profile of Rounded Cuts][Performance Profile \\ of Rounded Cuts]{\label{fig:max-cut-cuts-tol-0.0001}\includegraphics[width=0.28\linewidth]{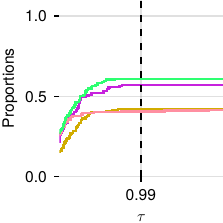}}%
\hfill
\subfloat{\includegraphics[width=0.15\linewidth]{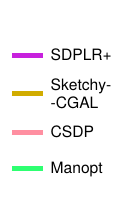}} \\
\centering
\subfloat[Runtime by Problem Size][Runtime by \\Problem Size]{\label{fig:max-cut-scalability-tol-1e-6}\includegraphics[width=0.28\linewidth]{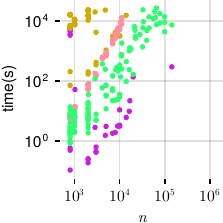}}%
\hfill
\subfloat[Performance Profile of Runtime][Performance Profile \\ of Runtime]{\label{fig:max-cut-running-time-tol-1e-6}\includegraphics[width=0.28\linewidth]{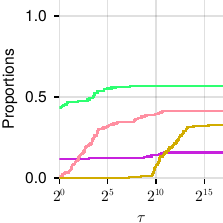}}%
\hfill
\subfloat[Performance Profile of Rounded Cuts][Performance Profile \\ of Rounded Cuts]{\label{fig:max-cut-cuts-tol-1e-6}\includegraphics[width=0.28\linewidth]{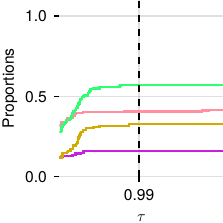}}%
\hfill
\subfloat{\includegraphics[width=0.15\linewidth]{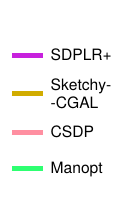}}
\caption{Results on Max Cut for $\varepsilon=10^{-4}$ (top) and
$\varepsilon=10^{-6}$ (bottom).  We repeat the experiments in
\Cref{sec:exp-max-cut} for each solver on all 202 graphs but with
higher precision. \texttt{SCAMS} is not included because of
runtime problems for high precision solves. We observe that
\texttt{SDPLR+} is quite competitive when
$\varepsilon = 10^{-4}$.}
  \label{fig:high-precision-max-cut}
\end{figure}

\begin{figure}[p]
\captionsetup[subfigure]{justification=centering}
\centering
\subfloat[Runtime by Problem Size][Runtime by \\Problem Size]{\label{fig:minimum-bisection-scalability-tol-0.0001}\includegraphics[width=0.28\linewidth]{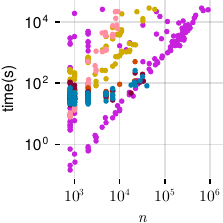}}%
\hfill
\subfloat[Performance Profile of Runtime][Performance Profile \\ of Runtime]{\label{fig:minimum-bisection-running-time-tol-0.0001}\includegraphics[width=0.28\linewidth]{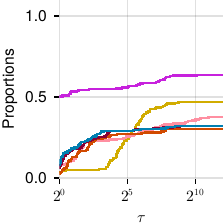}}%
\hfill
\subfloat[Performance Profile of Rounded Cuts][Performance Profile \\ of Rounded Cuts]{\label{fig:minimum-bisection-cuts-tol-0.0001}\includegraphics[width=0.28\linewidth]{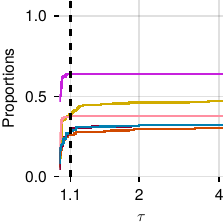}}%
\hfill
\subfloat{\includegraphics[width=0.15\linewidth]{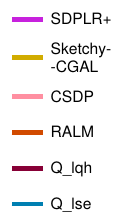}} \\
\subfloat[Runtime by Problem Size][Runtime by \\Problem Size]{\label{fig:minimum-bisection-scalability-tol-1e-6}\includegraphics[width=0.28\linewidth]{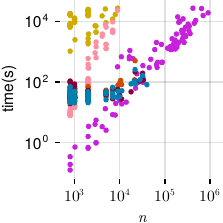}}%
\hfill
\subfloat[Performance Profile of Runtime][Performance Profile \\ of Runtime]{\label{fig:minimum-bisection-running-time-tol-1e-6}\includegraphics[width=0.28\linewidth]{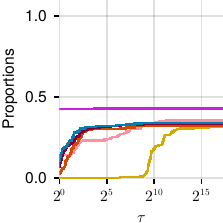}}%
\hfill
\subfloat[Performance Profile of Rounded Cuts][Performance Profile \\ of Rounded Cuts]{\label{fig:minimum-bisection-cuts-tol-1e-6}\includegraphics[width=0.28\linewidth]{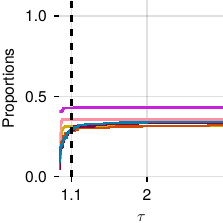}}%
\hfill
\subfloat{\includegraphics[width=0.15\linewidth]{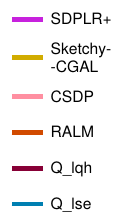}}%
\caption{Results on Minimum Bisection for higher precision
$\varepsilon = 10^{-4}$ (top) and $\varepsilon = 10^{-6}$ (bottom). We repeat the
experiments in \Cref{sec:exp-minimum-bisection} for each solver on
all 202 graphs but with higher precision. We include the results
of \texttt{RALM}, \texttt{\Qlqh}, \texttt{\Qlse} for conceptual
comparison but note that they have a quite different and
complicated set of termination conditions.  We did not touch their
code and the results of them are identical to those shown in
\Cref{fig:minimum-bisection}. We can see that \texttt{SDPLR+} has
a clear advantage on this problem even if precision is high.}
  \label{fig:high-precision--minimum-bisection}
\end{figure}

\begin{figure}[p]
\captionsetup[subfigure]{justification=centering}
\centering
\subfloat[Lov\'{a}sz Theta \\ Runtime by Prob.~Size][Lov\'{a}sz Theta \\ Runtime by \\ Problem Size]{\includegraphics[width=0.21\linewidth]{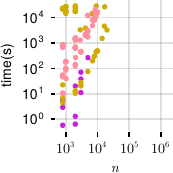}}%
\hfill
\subfloat[Lov\'{a}sz Theta \\Runtime Perf.~Prof.][Lov\'{a}sz Theta \\Runtime Perf. Prof.]{\includegraphics[width=0.21\linewidth]{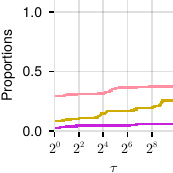}}%
\hfill
\subfloat[Cut Norm \\ Runtime by Prob.~Size][Cut Norm \\ Runtime by \\ Problem Size]{\includegraphics[width=0.21\linewidth]{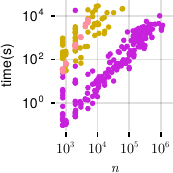}}%
\hfill
\subfloat[Cut Norm \\ Runtime Perf.~Prof.][Cut Norm \\ Runtime Perf. Prof.]{\includegraphics[width=0.21\linewidth]{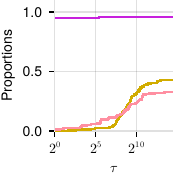}}%
\hfill
\subfloat{\includegraphics[width=0.12\linewidth]{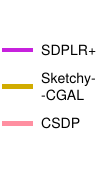}} \\
\subfloat[Lov\'{a}sz Theta \\ Runtime by Prob.~Size][Lov\'{a}sz Theta \\ Runtime by \\ Problem Size]{\includegraphics[width=0.21\linewidth]{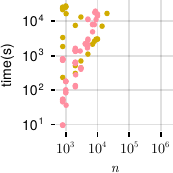}}%
\hfill
\subfloat[Lov\'{a}sz Theta \\Runtime Perf.~Prof.][Lov\'{a}sz Theta \\Runtime Perf. Prof.]{\includegraphics[width=0.21\linewidth]{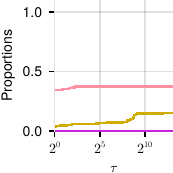}}%
\hfill
\subfloat[Cut Norm \\ Runtime by Prob.~Size][Cut Norm \\ Runtime by \\ Problem Size]{\includegraphics[width=0.21\linewidth]{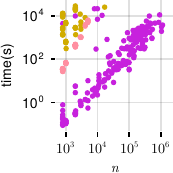}}%
\hfill
\subfloat[Cut Norm \\ Runtime Perf.~Prof.][Cut Norm \\ Runtime Perf. Prof.]{\includegraphics[width=0.21\linewidth]{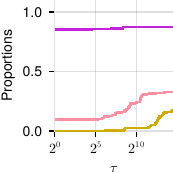}}%
\hfill
\subfloat{\includegraphics[width=0.12\linewidth]{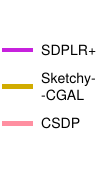}}
\caption{Results on Lov\'{a}sz Theta and Cut Norm when $\varepsilon=10^{-4}$ (top) and $\varepsilon=10^{-6}$ (bottom).}
  \label{fig:high-precision-lovasz-theta-cut-norm} 
\end{figure}

\subsubsection{Impact of Rank Parameter and Sketch Size}
\label{sec:fixed-rank-study}
Recall that \texttt{SDPLR+} and \texttt{Manopt} have a parameter $r$
that controls the rank of the low-rank factors they optimize, and
\texttt{SketchyCGAL} has a parameter $R$ that controls the size of the
sketch $\mS \in \R^{n \times R}$ they store.  We compare the performance of
\texttt{SDPLR+}, \texttt{Manopt} and \texttt{SketchyCGAL} on Max Cut
with various rank parameters and sketch sizes and explore how these
parameters affect their convergence speed to moderately accurate
solutions.  For simplicity, we let the rank parameter and the sketch
size be the same. We slightly modified the Max Cut code of
\texttt{Manopt} to achieve different ranks. The results are summarized
in \cref{fig:rank-sketch-size}.

\begin{figure}[t]
\captionsetup[subfigure]{justification=centering}
\centering
\subfloat[Runtime by Size \\ Rank 50]{\label{fig:rank-sketch-size-scalability}\includegraphics[width=0.24\linewidth]{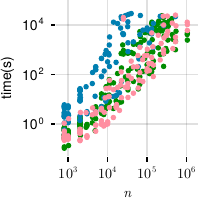}}%
\hfill
\subfloat[Runtime Perf. Prof. Rank 50]{\label{fig:rank-sketch-size-performance-plot}\includegraphics[width=0.24\linewidth]{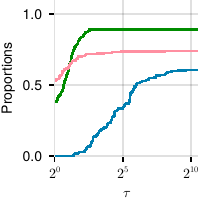}}%
\hfill
\subfloat[Runtime Perf. Prof. for both 10, 50]{\label{fig:rank-sketch-size-performance-plot-10-50}\includegraphics[width=0.24\linewidth]{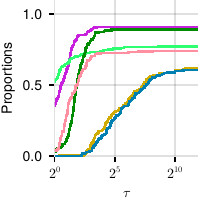}}%
\hfill
\subfloat{\includegraphics[width=0.12\linewidth]{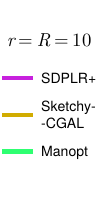}}%
\hfill
\subfloat{\includegraphics[width=.12\linewidth]{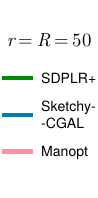}}%
\caption{Results on investigation of rank parameters and sketch
sizes.  We observe that if we decrease the rank parameter used
in \texttt{Manopt}, it has comparable performance against
\texttt{SDPLR+}, although \texttt{SDPLR+} remains the fastest
solver. When $r = R = 10$, the runtime by size and the runtime
performance plot are similar to
\cref{fig:rank-sketch-size-scalability} and
\cref{fig:rank-sketch-size-performance-plot}, so we omit them.
  From the performance plot
\cref{fig:rank-sketch-size-performance-plot-10-50}, we can see
that in general, a higher rank parameter or sketch size will
result in longer running time for moderate accuracy.}
\label{fig:rank-sketch-size}
\end{figure}

\clearpage
\section{Data Science Applications}
\Cref{sec:experiments} benchmarks \SDPLRp against many other solvers
on four classic graph-related SDPs.  In this section, we present a brief study applying \SDPLRp to broader data science
applications, including graph low-conductance cuts (\Cref{sec:graph-low-cond}), matrix completion (\Cref{sec:matrix-completion}) and $k$-means clustering (\Cref{sec:kmeans}).

\subsection{Graph Low Conductance Cuts}
\label{sec:graph-low-cond}

Low-conductance cut is a popular objective in network science
\citep{DBLP:journals/jct/AlonM85}, which finds a vertex set $S$ that
minimizes
\begin{align*}
\phi(S) \defeq \frac{\cut(S, \bar{S})}{\min \{\vol(S), \vol(\bar{S})\}}.
\end{align*}
By dividing the cut size by the volume of the smaller side, this
objective yields a more balanced, sparse cut than the simple minimum
cut objective. Recently, Huang, Seshadhri, and Gleich studied a
size-specific variant of conductance, called $\mu$-conductance, that
disregards sets with volume smaller than $\mu \vol(G)$
\cite{HSGTheoretical2023,huang23:_cheeg_inequal_size_specif_conduc}. In
other words,
\begin{align*}
\phi_{\mu}(G) = \min_{S} \phi(S) \quad \text{subject to} \quad \mu \vol(G) \leq \vol(S) \leq (1 - \mu) \vol(G).
\end{align*}
By varying $\mu$, $\mu$-conductance can reveal a richer cut structure of
the graph, and when $\mu = 0$, it reduces to the standard conductance.
However, $\mu$-conductance is $\mathsf{NP}$-hard to compute; existing
approximation algorithms can only give upper bounds. To bound
$\mu$-conductance from below, Huang, Seshadhri, and Gleich proposed
the following SDP relaxation.
\begin{align}
\label{eq:mu-cond}
\tag{$\mu$-conductance}
\minimize_{\mX \succeq 0} \quad & \Tr(\mL\mX) \\
\text{subject to} \quad& \Tr(\mD\mX) = 1 \nonumber \\
& \vd \mX \vd^T = 0 \nonumber\\
& \diag(\mX) \le \frac{1 - \mu}{\mu \vol(G)} \ones \nonumber\\
& \diag(\mX) \geq \frac{\mu}{(1 - \mu) \vol(G)} \ones. \nonumber
\end{align}
This relaxation gives a lower bound for $\mu$-conductance. Together with
the upper bound, one can sandwich the true $\mu$-conductance.

To handle the inequalities in this SDP relaxation, we extend our
augmented Lagrangian to the classical form for inequality-constrained
problems. We still use L-BFGS to find stationary points, but replace
exact line search with Armijo line search due to nonsmoothness.

We compare the lower bounds computed by \SDPLRp with the results from
\citet{HSGTheoretical2023}. For a fair comparison, we terminate the
solver based on the absolute tolerances of \citet{HSGTheoretical2023}.
Specifically, our primal infeasibility termination condition for
equality constraints $\mathcal{A}(\mX) = \vb$ is $\norm{\mathcal{A}(\mY\mY^T) - \vb}_{\infty} \leq
10^{-5}$, and for inequality constraints $\mathcal{A}(\mX) \leq \vb$ is
$\norm{\max(\mathcal{A} (\mY\mY^T) - \vb, 0)}_{\infty} \leq 10^{-5}$. Unlike
\citet{HSGTheoretical2023}, since \SDPLRp supports termination based
on suboptimality gap rather than purely on stationarity, we set the 
relative suboptimality tolerance to $10^{-1}$. Moreover, we scale all constraints so that
the left-hand side matrices $\mA_i$ satisfy $\|\mA_i\|_F = \|\mD\|_{F}$;
this change results in a tighter feasible region in practice.
Moreover, we also provide a comparison against \SketchyCGAL.  Because
\SketchyCGAL does not support early termination when the SDP instances
contain inequalities, we turn \eqref{eq:mu-cond} into the 
canonical form without inequalities. In particular, we introduce nonnegative slack variables $\vs, \vt$ such that
\begin{align*}
  &\diag(\mX) + \vs = \frac{(1 - \mu) \ones}{\mu \vol(G)} \\
  &                   \vs + \vt = \frac{(1 - \mu) \ones}{ \mu \vol(G)} - \frac{\mu \ones}{(1 - \mu) \vol(G)} 
\end{align*}
Putting them onto the diagonal of the decision matrix $\mX$ gives the following SDP,
\begin{align}
  \label{eq:mu-cond-eq}
  \tag{$\mu$-conductance-equality}
  \minimize_{\mX \succeq 0} & \quad \langle
                        \begin{bmatrix}
           \mL & \zeros & \zeros  \\
           \zeros &  \zeros & \zeros \\
           \zeros &  \zeros & \zeros                           
                        \end{bmatrix},
         \mX \rangle \\ 
  \subjectto & \quad \langle
 \begin{bmatrix}
           \mD & \zeros & \zeros  \\
           \zeros & \zeros & \zeros  \\
           \zeros & \zeros & \zeros 
 \end{bmatrix},
               \mX \rangle = 1 \nonumber \\
  & \quad
  \langle
         \begin{bmatrix}
           \vd \vd^T & \zeros & \zeros  \\
           \zeros & \zeros & \zeros \\
           \zeros & \zeros & \zeros 
        \end{bmatrix}, 
        \mX \rangle = 0 ,
 \nonumber  \\
                      & \quad \langle
                        \begin{bmatrix}
           \ve_i \ve_i^T & \zeros & \zeros  \\
           \zeros & \ve_i \ve_i^T & \zeros \\
           \zeros & \zeros & \zeros                           
                        \end{bmatrix},
        \mX \rangle = \frac{1 - \mu}{ \mu \vol(G) } , \quad \forall i \in [n] \nonumber \\
 & \quad  \langle 
   \begin{bmatrix}
           \zeros & \zeros & \zeros \\
           \zeros & \ve_i \ve_i^T & \zeros \\
           \zeros & \zeros & \ve_i \ve_i^T 
   \end{bmatrix}
     ,
         \mX \rangle = \frac{1 - 2\mu}{\mu(1 - \mu) \vol(G)}, \quad \forall i \in [n]\nonumber 
\end{align} where $\ve_i$ is the $i$-th standard basis and the
decision variable $\mX$ is of size $3n \times 3n$.  Both \eqref{eq:mu-cond}
and \eqref{eq:mu-cond-eq} have explicit trace bounds that can be
derived from constraints. For \SketchyCGAL, we also set the relative
suboptimality tolerance to $10^{-1}$, same with \SDPLRp.

The results are summarized in \cref{fig:mu-conductance}. For the Network Community Profile (NCP) plots, following \cite{HSGTheoretical2023}, we smooth the lower bounds at different $\mu$s and make sure the curve for each solver is non-decreasing. Part of the lower bound curve for \SketchyCGAL is missing because the returned lower bounds are negative, which is meaningless for lower bounding $\mu$-conductance. In the scalability and time performance profile, we consider any solve longer than 8 hours or not reaching desired accuracy failed, even if it provides a non-trivial lower bound.

We remark that the original lower bound from
\citet{HSGTheoretical2023} was computed by adding $\alpha
\min\setof{\lambda_{\min} (\mC - \cA^*(\vlambda)), 0}$ to the current primal
value $\langle\mL, \mX\rangle$, which implicitly assumes that $\mX$ has satisfied
all KKT conditions. We reran those experiments with the more rigorous
lower bound
\[\vlambda^T \vb + \alpha \min\setof{\lambda_{\min} (\mC - \cA^*(\vlambda)), 0},\]
so our lower bound curves may fall below those in
\citet{HSGTheoretical2023}.

\begin{figure}[t]
    \captionsetup[subfigure]{justification=centering}
  \centering
  \subfloat[\textsc{ca-AstroPh}]{\label{fig:ca-astroph}\includegraphics[width=0.25\linewidth]{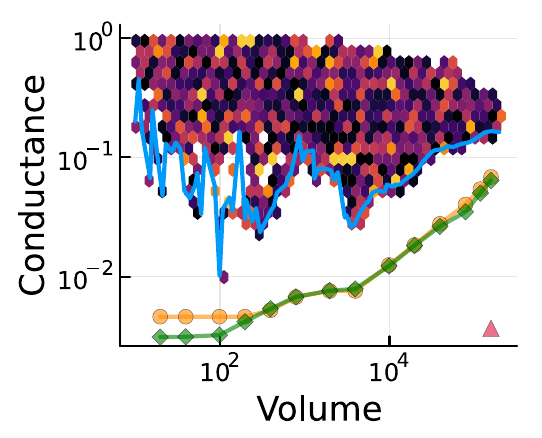}}%
  \hfill  \subfloat[\textsc{ca-HepPh}]{\label{fig:ca-hepph}\includegraphics[width=0.25\linewidth]{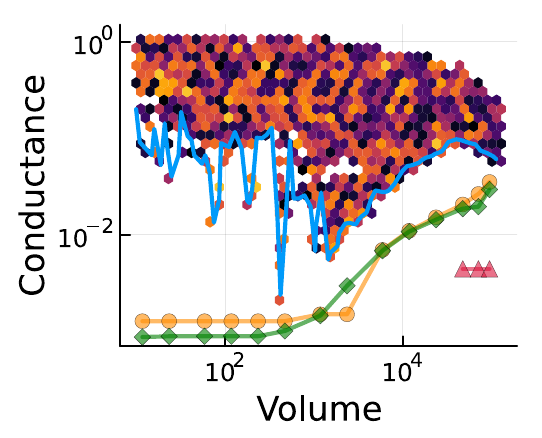}}%
  \hfill  \subfloat[\textsc{email-Enron}]{\label{fig:email-Enron}\includegraphics[width=0.25\linewidth]{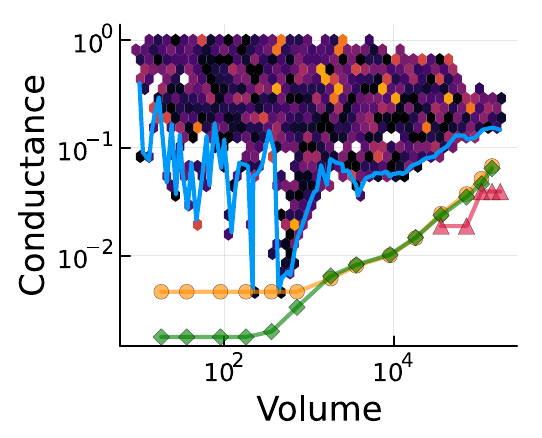}}%
  \hfill  \subfloat[\textsc{Scalability}]{\label{fig:mu-cond-scal}\includegraphics[width=0.25\linewidth]{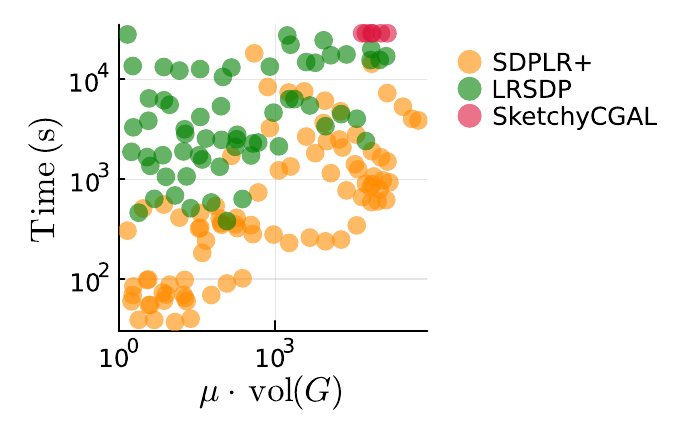}}%
  \\ \subfloat[\textsc{deezer}]{\label{fig:deezer}\includegraphics[width=0.25\linewidth]{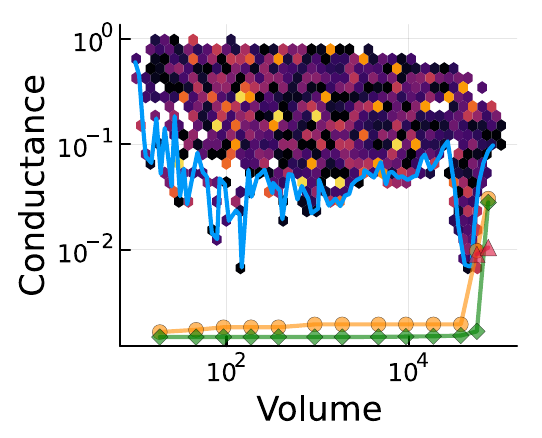}}%
  \hfill  \subfloat[\textsc{facebook}]{\label{fig:facebook}\includegraphics[width=0.25\linewidth]{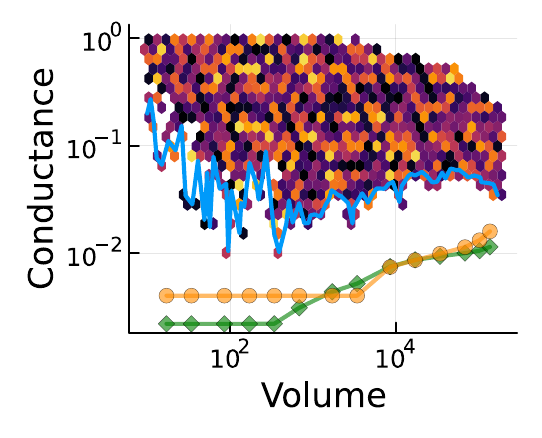}}%
  \hfill  \subfloat[\textsc{dblp}]{\label{fig:dblp}\includegraphics[width=0.25\linewidth]{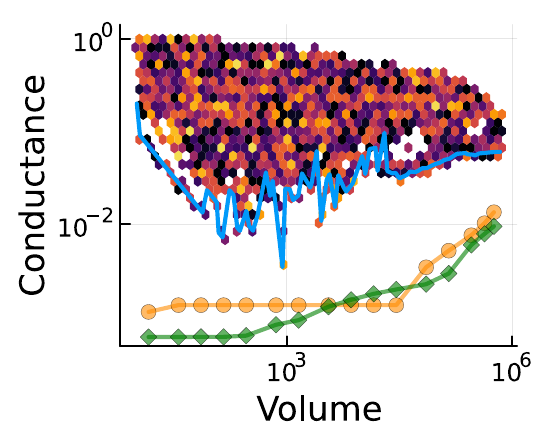}}%
  \hfill  \subfloat[\textsc{time profile}]{\label{fig:mu-cond-time-profile}\includegraphics[width=0.25\linewidth]{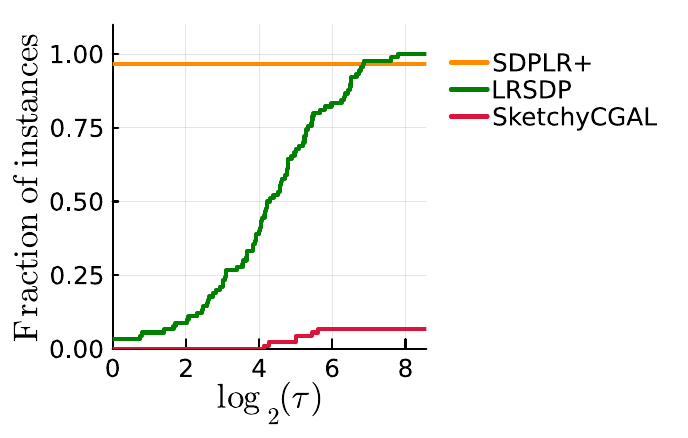}}%
  \caption{Comparison of \SDPLRp (yellow) and \LRSDP (green), the
solver from \citet{HSGTheoretical2023} for \eqref{eq:mu-cond}, and
\SketchyCGAL (red), on six social networks. The largest graph,
\textsc{dblp}, contains roughly 200K vertices and 700K edges, which
gives an SDP with $\mX$ of size $2e5 \times 2e5$ and $\nnz(\mL) = 7e5$.
\LRSDP also uses the Burer-Monteiro factorization, but handles
inequalities via slack variables and L-BFGS-B, unlike \SDPLRp which
handles them directly via ALM. \LRSDP also does not dynamically adjust
the rank or terminate based on the duality gap.  The heatmaps show the
conductances of $10^5$ cuts found by ACL localized PageRank
\citep{DBLP:conf/focs/AndersenCL06}. The blue curves are the
smallest-conductance cuts at each volume, giving upper bounds on the
true $\mu$-conductance.  The yellow, green and red curves are lower
bounds given by \SDPLRp, \LRSDP and \SketchyCGAL
respectively. \Cref{fig:mu-cond-scal} shows the scalability of the two
solvers (running time vs.\ $\mu \times \vol(G)$), and
\Cref{fig:mu-cond-time-profile} shows their performance profiles.
Apart from a few failures, \SDPLRp is much faster than \LRSDP and
\SketchyCGAL on most runs. Since we scale the constraints in
\eqref{eq:mu-cond}, \SDPLRp often yields tighter lower bounds, as
visible in the curves.}
\label{fig:mu-conductance}
\end{figure}

\subsection{Matrix Completion SDP}
\label{sec:matrix-completion}

Matrix completion recovers a low-rank matrix $\mM \in \R^{n \times p}$
from a subset $\Omega$ of its observed entries, where $m = |\Omega|$
denotes the number of observations. A canonical application
is collaborative filtering, where $\mM$ contains user-item ratings and
most entries are missing \citep{Koren2009}.
\citet{CandesRecht2009, CandesTao2010} showed that, under mild conditions on the
distribution of the missing entries, minimizing the nuclear norm $\|\mX\|_*$
subject to the observed entries recovers $\mM$ exactly. This can be cast as the following
SDP~\citep{RechtFazelParrilo2010}:
\begin{align}
  \label{eq:matrix-completion-sdp}
  \tag{Matrix Completion}
  \minimize_{\mX \in \R^{n \times p}}
  \quad & \tfrac{1}{2}\bigl(\Tr(\mW_1) + \Tr(\mW_2)\bigr) \\
  \text{subject to} \quad &
  \begin{pmatrix} \mW_1 & \mX \\ \mX^T & \mW_2 \end{pmatrix} \succeq 0
  \nonumber \\
  & \mX_{ij} = \mM_{ij} \quad \forall\, (i, j) \in \Omega. \nonumber
\end{align}
The objective minimizes $\|\mX\|_*$ because $\|\mX\|_* =
\min_{\mW_1,\mW_2} \tfrac{1}{2}(\Tr(\mW_1) + \Tr(\mW_2))$ subject
to the block PSD constraint.
When $\Omega$ is sampled uniformly at random, $m = O(Nr\log^{O(1)} N)$
observations suffice for exact recovery with high probability, where $r$ is the rank of $\mM$
and $N = \max\{n, p\}$ \citep{CandesTao2010}.

We apply \SDPLRp to \eqref{eq:matrix-completion-sdp}. Since
\eqref{eq:matrix-completion-sdp} does not include an explicit trace
bound, we use the trace bound
$ 2\sqrt{\min\{n, p\}} \|\mM_{\Omega}\|_F$, where
$\mM_{\Omega}$ is the matrix with the observed entries in $\Omega$ and zeros
elsewhere. This is a valid trace bound because the objective of \eqref{eq:matrix-completion-sdp}
is $\|\mX\|_{*}$, the optimal $\mX$ must have a trace upper bounded by any feasible solution.
As a result the optimal $\mX$ has trace at most $\|\mM_{\Omega}\|_{*} \leq 2\sqrt{\min\{n, p\}} \|\mM_{\Omega}\|_F $. We generate synthetic square rank-$r$ matrices
$\mM \in \R^{n \times n}$ via the factorization $\mM = \mU\mV^T$, where
$\mU, \mV \in \R^{n \times r}$ have i.i.d.\ $\mathcal{N}(0, 1/r)$ entries, and draw
$m$ observations by uniform sampling without replacement.

We test six values of $n$ ranging from $10^4$ to $10^6$, ranks $r \in \{1, 3, 5\}$,
and ratios $m/n \in \{10, 50, 100, 500, 1000\}$. Excluding cases where $m > 10^7$,
this yields 60 SDPs in total. We set both the primal infeasibility and suboptimality tolerances to $10^{-3}$ and give each problem instance four hours.

On this problem, we mainly compare against \SketchyCGAL as \CSDP runs into memory issues for these large-scale SDPs. It turns out \SketchyCGAL fails to achieve the desired $10^{-3}$ accuracy on all instances. The results are summarized in \Cref{fig:matrix-completion}.
\begin{figure}[t]
    \captionsetup[subfigure]{justification=centering}
  \centering
  \subfloat[$r = 1$]{\label{fig:mc_time_rank_1}\includegraphics[width=0.28\linewidth]{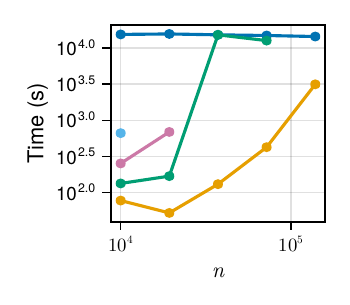}}
  \subfloat[$r = 3$]{\label{fig:mc_time_rank_3}\includegraphics[width=0.28\linewidth]{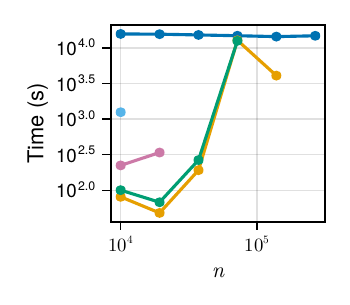}}
  \subfloat[$r = 5$]{\label{fig:mc_time_rank_5}\includegraphics[width=0.28\linewidth]{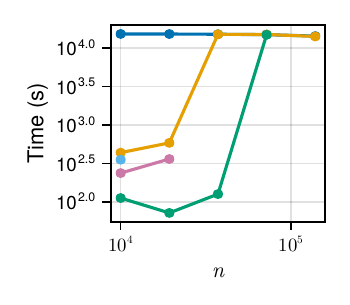}}
\subfloat{\label{fig:mc_time_rank_5}\includegraphics[width=0.15\linewidth]{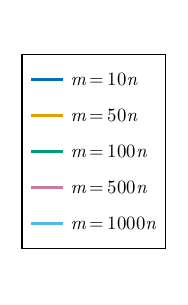}}
  \\
  \subfloat[$r = 1$]{\label{fig:mc_relerr_rank_1}\includegraphics[width=0.28\linewidth]{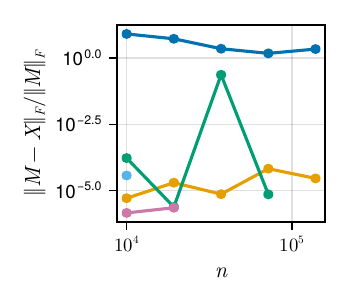}}
  \subfloat[$r = 3$]{\label{fig:mc_relerr_rank_3}\includegraphics[width=0.28\linewidth]{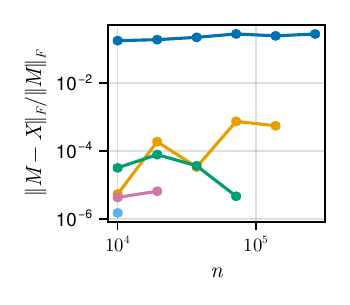}}
  \subfloat[$r = 5$]{\label{fig:mc_relerr_rank_5}\includegraphics[width=0.28\linewidth]{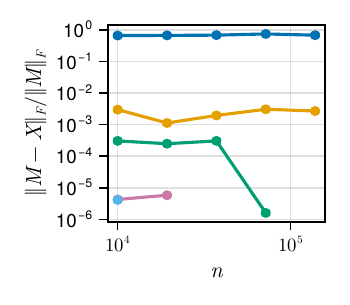}}
\subfloat{\label{fig:mc_time_rank_5}\includegraphics[width=0.15\linewidth]{figs/mc_legend.pdf}}
\caption{\SDPLRp applied to \eqref{eq:matrix-completion-sdp} on
  synthetic data. \SketchyCGAL fails to achieve the desired $10^{-3}$
  accuracy on all problem instances. The first row shows running time
  and the second row shows relative error
  $\|\mM - \mX\|_F / \|\mM\|_F$, for each rank $r \in \{1, 3, 5\}$.  Eight
  of the 60 instances ran out of RAM and are excluded.  When $m/n$ is
  too small, \SDPLRp fails to recover $\mM$ exactly, consistent with
  the theory. Once $m/n$ exceeds a threshold (around $100$ here), the
  running time scales with $m$. Higher rank requires slightly more
  time and more observations for reliable recovery. More observations
  also lead to smaller relative error $\|\mM - \mX\|_F /
  \|\mM\|_F$. \SDPLRp takes 354 seconds on
  $(n, m/n, r) = (10^4, 10^3, 5)$ and 3145 seconds on
  $(n, m/n, r) = (138950, 50, 1)$.}
\label{fig:matrix-completion}
\end{figure}

\section{Limitations and Conclusions}
\label{sec:conclusion}

\paragraph{Limitations} Empirical studies of solvers are always
challenging. Do we pick optimal parameters for each solver via
hyperparameter search or use default options? In this case,
we chose to use default options. To address some of the limitations
that this imposes, we did a fixed rank study in
\cref{sec:fixed-rank-study} that compares the solvers in a directly
comparable parameter regime.  Performance profile plots also can be
misleading~\cite{Gould2016}. This is why we show both runtime and size
as well -- which typically show the same trend. For one recent
solver, we were unable to obtain an implementation to
compare~\cite{monteiro2024low} and data from the paper is from a very
different accuracy regime. For another recent
solver, \citet{angell2024fast}, we encountered issues when trying to
complete a full evaluation. In this case, we present a comparison
against values from their paper in \cref{sec:extra-solvers}.

Although we target \emph{trace-bounded SDPs}, this is a very general
setting as many problems have easy-to-compute bounds on the maximum
possible trace of the solution. For those that don't, doubling
approaches could be used to scale the upper bound through a variety of
choices.

\paragraph{Conclusion}
The research field has made tremendous progress on solving large
SDPs. This paper continues this thread of research and proposes a new
solver for trace-bounded SDPs that adds a suboptimality bound to
\texttt{SDPLR} to produce the improved solver \texttt{SDPLR+}. We
were impressed by the gains from this straightforward
modification. This further supports the idea that there
remain significant potential speed improvements in SDP solvers.

\newpage
\appendix

\section{Connections between Different Suboptimality Bounds}
\label{sec:suboptimality-other}
For completeness, we include a concise proof for the surrogate duality
bound from~\citet{YTF+Scalable2021a}. We slightly adjust their proof
because the compact domain we consider is
$\Delta_{\alpha} = \{\mX \succeq 0: \Tr(\mX) \leq \alpha\}$, which is slightly different from
$\{\mX \succeq 0: \Tr(\mX) = \alpha\}$. The adjustment is straightforward. For
convenience, we let $p(\mX) \defeq \mathcal{A}(\mX) - \vb$ denote the primal
violation of $\mX$.

They consider minimizing the following augmented
Lagrangian
\begin{align}
\label{eq:20}
  f(\mX) = \langle \mC, \mX \rangle - \vlambda^T p(\mX) + \frac{\sigma}{2} \|p(\mX)\|^2 
\end{align}
with $\sigma > 0$ over $\Delta_{\alpha}$. Let
$\vlambda' = \vlambda - \sigma p(\mX)$, we observe that 
$\nabla f(\mX) = \mC - \mathcal{A}^*(\vlambda')$ and as a result
\begin{align}
  \label{eq:31}
  \langle \nabla f(\mX), \mX \rangle - f(\mX) &= - (\vlambda')^T \cA(\mX) + \vlambda^T p(\mX) - \frac{\sigma}{2} \| p(\mX)\|^2 \\
                             &=  - (\vlambda')^T \cA(\mX) + (\vlambda' + \sigma p(\mX))^T p(\mX) - \frac{\sigma}{2} \| p(\mX)\|^2 \nonumber\\
   &= - \vb^T \vlambda' + \frac{\sigma}{2} \| p(\mX)\|^2 \nonumber.
\end{align}

Because \cref{eq:20} is convex, we have the
following surrogate duality bound by applying the surrogate duality
bound in Frank-Wolfe~\cite{JagRevisiting} to $f$,
\begin{align}
  \label{eq:24}
  \max_{\mH \in \Delta_{\alpha}} \langle \nabla f(\mX), \mX - \mH \rangle & \geq f(\mX) - \min_{\mZ \in \Delta_{\alpha}} f(\mZ) \\
                                             &  \geq f(\mX) - f(\mX^{*}) 
                                               = f(\mX) - \langle\mC, \mX^{*} \rangle  \nonumber
\end{align}
where \(\mX^*\) is one optimal solution for \cref{eq:3}.  The second
inequality holds because
\(f(\mX^*) \geq \min_{\mZ \in \Delta_{\alpha}} f(\mZ) \) and the last equality holds due
to \(\mX^*\) is feasible, i.e. $p(\mX^{*}) = 0$.

Since optimizing a linear function over $\Delta_{\alpha}$ can be turned into evaluating one extreme
eigenvalue, we have 
\begin{align}
  \label{eq:30}
\max_{\mH \in \Delta_{\alpha}}^{}  \langle \nabla f(\mX), -\mH \rangle 
& = 
-\min_{\mH \in \Delta_{\alpha}}^{}  \langle \nabla f(\mX), \mH \rangle 
   = - \alpha \min_{}^{} \{0, \lambda_{\min} ( \nabla f(\mX))\} \\
  & = - \alpha \min_{}^{} \{0, \lambda_{\min} ( \mC - \cA^{*}(\vlambda'))\}\nonumber.
\end{align}
Hence we can bound the suboptimality by 
\begin{align}
  \label{eq:21}
  \langle\mC, \mX - \mX^{*}\rangle 
  & \leq\langle\mC, \mX\rangle +  \max_{\mH \in \Delta_{\alpha}} \langle \nabla f(\mX), \mX - \mH \rangle - f(\mX) \\
  &= \langle \mC, \mX \rangle - \vb^{T} \vlambda'+ \frac{\sigma}{2} \|p(\mX)\|^2 
    - \alpha \min\{\lambda_{\min} (\mC - \cA^*(\vlambda')), 0\} \nonumber 
\end{align}
where the first relation is due to~\cref{eq:24}, and in the second relation
we use \cref{eq:31} and \cref{eq:30}.

We can see \cref{eq:21} has a clear relation to \cref{eq:14} as
\(\vlambda'\) is feasible for \cref{eq:12}. Also we notice by
\cref{eq:14}, the term $\sigma \|p(\mX)\|^2 / 2$ is unnecessary and
will result in a worse bound. Interestingly, $\vlambda'$ is exactly
the updated dual in the classical augmented Lagrangian framework.

\section{Comparison against \texttt{USBS}}
\label{sec:extra-solvers}
While we encountered issues when trying to run a full comparison
against \texttt{USBS}, we are able to do a preliminary comparison
against the data reported in their paper. Specifically, we extract
data from their performance figures via the tool
WebPlotDigitizer~\cite{WebPlotDigitizer}. The results are shown in
\cref{tab:usbs-compare}. We were unable to learn if their solver was
run with any multithreading. It was also run with a lower accuracy
than we run $\texttt{SDPLR+}$. Nonetheless, we see that
\texttt{SDPLR+} is faster on most problems.


\begin{table}
\footnotesize
\caption{A comparison against \texttt{USBS}~\cite{angell2024fast}
  using results from their paper. These results indicate that \texttt{SDPLR+} is
  faster on many problems by a considerable margin, albeit these are
  from different systems.  }
    \begin{tabularx}{\linewidth}{@{}lXXXXXXXX@{}}
      \toprule
      Problem/Time (s) & {fe\_sphere} & hi2010 & fe\_body & me2010 \\
      \midrule
      \texttt{USBS/warm} $(\varepsilon = 10^{-1})$ & 31 & 125 & 273 & 417 \\
      \texttt{SDPLR+} $(\varepsilon = 10^{-2})$ & 7 & 160 & 38 & 613  \\
      \midrule
       Problem/Time (s)                & fe\_tooth & 598a & 144 & auto \\
       \midrule 
      \texttt{USBS/warm} $(\varepsilon = 10^{-1})$ &  348 & 636 & 810 & 3444\\
      \texttt{SDPLR+} $(\varepsilon = 10^{-2})$ &  43 & 192 & 170 & 634 \\
      \bottomrule
    \end{tabularx}

    \label{tab:usbs-compare}
\end{table}

\section{$K$-Means Clustering SDP}
\label{sec:kmeans}

In this Section, we illustrate the potential applications of \SDPLRp on $K$-means clustering,
where SDP can play an important building block.

$K$-means clustering partitions $n$ data points
$\vx_1, \ldots, \vx_n \in \R^d$ into $k$ groups to minimize the total
within-cluster sum of squared distances to cluster centroids. It is
one of the most widely used methods in unsupervised learning, but the
underlying combinatorial optimization problem is
$\mathsf{NP}$-hard~\cite{Aloise_2009}. Many methods in practice get
stuck at local optima when starting from poor
initializations. \citet{PengWei2007} proposed a semidefinite
programming relaxation that not only gives a lower bound on
the $k$-means objective, but also provides good initializations
for practical $k$-means methods.

Let $\mW \in \R^{n \times d}$ be the data matrix, standardized so
that each feature has mean~$0$ and standard deviation~$1$.  The $k$-means SDP
relaxation of~\citet{PengWei2007}, dropping the $\mX^2 = \mX$
constraint for scalability, is
\begin{align}
  \label{eq:kmeans-sdp}
  \tag{$k$-means SDP}
  \maximize_{\mX \succeq 0} \quad &  \langle \mW \mW^T, \mX \rangle \\
  \text{subject to} \quad & \mX \ones = \ones \nonumber \\
                          & \Tr(\mX) = k \nonumber \\
                          & \mX_{ij} \geq 0 \quad \forall\, i, j. \nonumber
\end{align}
Here $\mX$ plays the role of a relaxed cluster-membership matrix: in
the integer program, $\mX_{ij} = 1/|C_\ell|$ if $i$ and $j$ belong to the
same cluster $C_\ell$, and $\mX_{ij} = 0$ otherwise.

We test \SDPLRp on two datasets, Iris~\citep{fisher1936iris} and Pen
Digits~\citep{pendigits_ucirepo}. We use the full Iris dataset with 3
classes and 50 data points per class. Pen Digits has 10 classes.
Because \eqref{eq:kmeans-sdp} is densely constrained, we sample 500
data points per class, giving a downsampled dataset of 5000 points. We
set the primal infeasibility tolerance to $10^{-4}$ and the
suboptimality tolerance to $10^{-2}$. The results are summarized in
\Cref{fig:k-means-sdp}. Although
    \eqref{eq:kmeans-sdp} does not directly improve accuracy, it can serve
    as a fundamental building block for more accurate and sophisticated
    $k$-means solvers~\citep{zhuang2022sketch}. We see that our package \SDPLRp can
    efficiently solve this $k$-means clustering SDP.

\begin{figure}[t]
    \captionsetup[subfigure]{justification=centering}
  \centering
  \subfloat[\textsc{Iris (true)}]{\label{fig:iris-true}\includegraphics[width=0.25\linewidth]{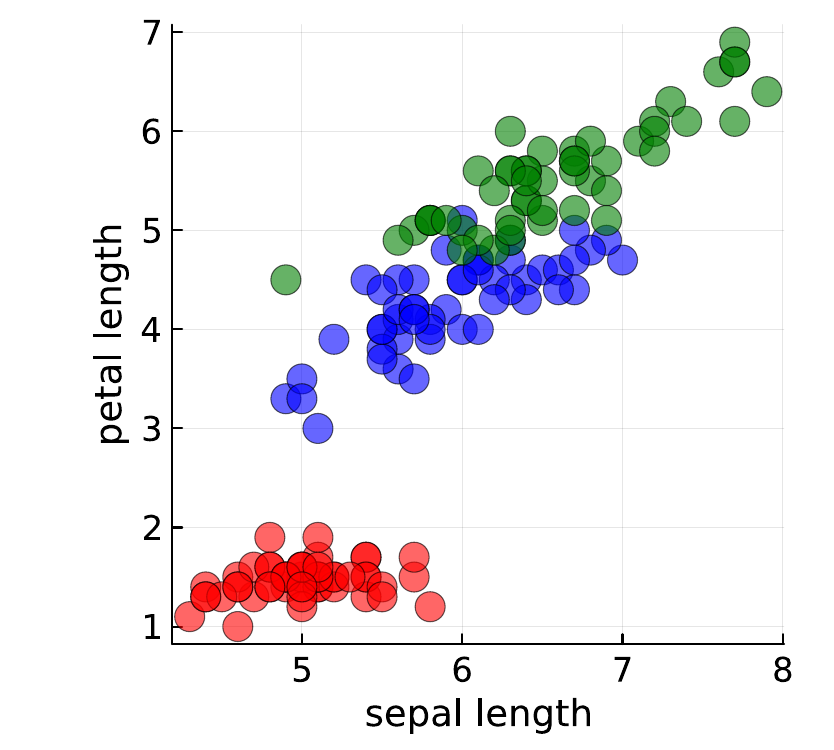}}%
  \hfill 
  \subfloat[\textsc{Iris (pred)}]{\label{fig:iris-pred}\includegraphics[width=0.25\linewidth]{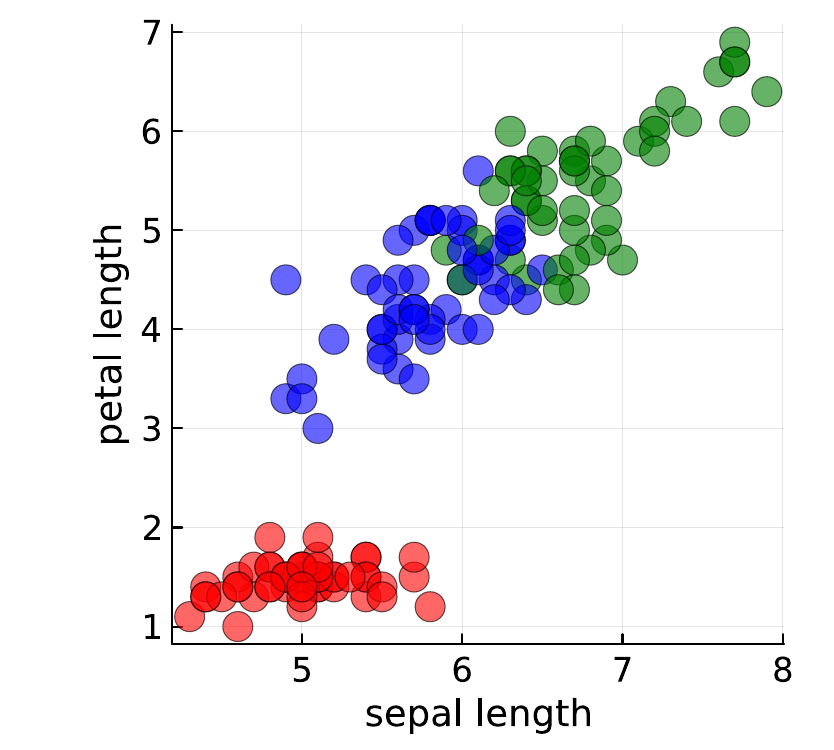}}%
  \hfill
  \subfloat[\textsc{Pen Digits (true)}]{\label{fig:pendigits-true}\includegraphics[width=0.25\linewidth]{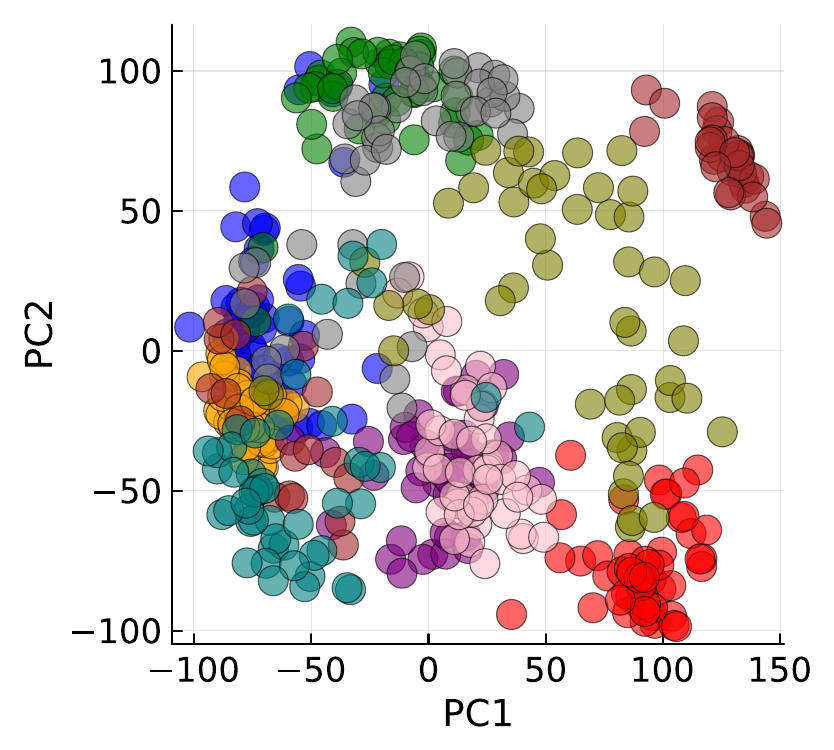}}%
  \hfill
  \subfloat[\textsc{Pen Digits (pred)}]{\label{fig:pendigits-pred}\includegraphics[width=0.25\linewidth]{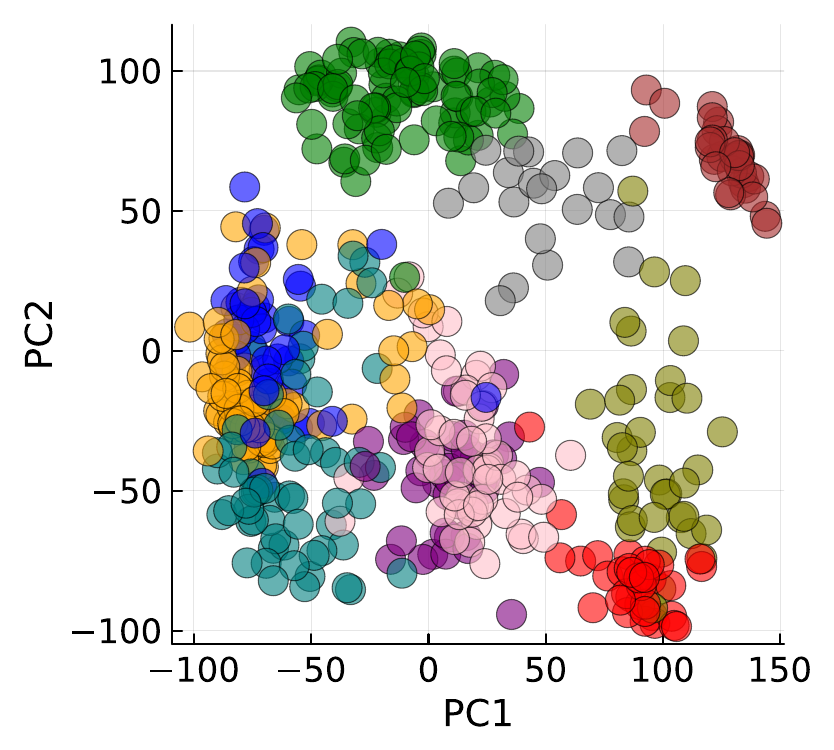}}%
  \caption{\SDPLRp applied to the $k$-means SDP~\eqref{eq:kmeans-sdp}
    on Iris and Pen Digits.  \Cref{fig:iris-true,fig:pendigits-true}
    show the ground-truth labels.
    \Cref{fig:iris-pred,fig:pendigits-pred} show the predicted
    clusters from running \kmeanspp on the spectral embeddings of the
    solution to~\eqref{eq:kmeans-sdp}. For the Iris plots, the axes
    show petal length and sepal length; for the Pen Digits plots, the
    axes show the leading two principal components. Directly running
    \kmeanspp achieves accuracies of 89.3 on Iris and 70.6 on Pen
    Digits; running \kmeanspp on the SDP embeddings gives 82.7 and
    71.0, respectively. The runs take 41 seconds (Iris) and 3540
    seconds (Pen Digits).}
\label{fig:k-means-sdp}
\end{figure}

\bibliographystyle{dgleich-bib3}
\bibliography{refs,yufan_dimacs}

\end{document}